\documentclass [12pt]{amsart}
\usepackage{amsmath}
\usepackage{amssymb}

\newtheorem{thm}{Theorem}[section]

\newtheorem{lem}[thm]{Lemma}
\newtheorem{prop}[thm]{Proposition}
\theoremstyle{definition}

\theoremstyle{remark}

\numberwithin{equation}{section}

\newcommand{\beas}{\begin{eqnarray*}}
\newcommand{\eeas}{\end{eqnarray*}}
\newcommand{\bes} {\begin{equation*}}
\newcommand{\ees} {\end{equation*}}
\newcommand{\be} {\begin{equation}}
\newcommand{\ee} {\end{equation}}
\newcommand{\bea} {\begin{eqnarray}}
\newcommand{\eea} {\end{eqnarray}}
\newcommand{\ra} {\rightarrow}
\newcommand{\txt} {\textmd}
\newcommand{\ds} {\displaystyle}
\newcommand{\la} {\lambda}

\begin{document}

\title[Segal-Bargmann transform and Paley-Wiener theorems on $\mathbb{HM}$]
 {Segal-Bargmann transform and Paley-Wiener theorems on Heisenberg motion groups}

\author{Suparna Sen}

\address{Department of Mathematics, Indian Institute of Science, Bangalore - 560012, India.}

\email{suparna@math.iisc.ernet.in.}

\thanks{The author was supported by Shyama Prasad Mukherjee Fellowship from Council of Scientific and Industrial Research, India.}


\begin{abstract}
We study the Segal-Bargmann transform on the Heisenberg motion
groups $\mathbb{H}^n \ltimes K,$ where $\mathbb{H}^n$ is the
Heisenberg group and $K$ is a compact subgroup of $U(n)$ such that
$(K,\mathbb{H}^n)$ is a Gelfand pair. The Poisson integrals
associated to the Laplacian for the Heisenberg motion group are
also characterized using Gutzmer's formulae. Explicitly realizing
certain unitary irreducible representations of $\mathbb{H}^n
\ltimes K,$ we prove the Plancherel theorem. A Paley-Wiener type
theorem is proved using complexified representations.
\vspace{0.1in}
\begin{flushleft}
\vspace*{0.1in}
\end{flushleft}
\end{abstract}

\maketitle

\section{Introduction}

For any $f \in L^2(\mathbb{R}^n),$ it is easy to see that
$f*\rho_t$ extends as an entire function to the whole of
$\mathbb{C}^n,$ where $\ds{\rho_t(x) = \frac{1}{(4\pi
t)^{\frac{n}{2}}} e^{\frac{-|x|^2}{4t}}}$ is the heat kernel on
$\mathbb{R}^n$ and the image of $L^2(\mathbb{R}^n)$ under the map
$\ds{f \rightarrow f * \rho_t}$ can be characterized as the
Hilbert space of entire functions on $\mathbb{C}^n$ which are
square integrable with respect to the positive weight
$\rho_{t/2}(y)dxdy.$

The mapping $\ds{f \rightarrow f * \rho_t}$ is called the
Segal-Bargmann transform, also known as the coherent state
transform or the heat kernel transform. Segal and Bargmann
independently proved in the 1960's in the context of quantum
mechanics that this transform is a unitary map from
$L^2(\mathbb{R}^n)$ onto $\ds{\mathcal{O}(\mathbb{C}^n)}$
$\ds{\bigcap L^2(\mathbb{C}^n, \mu)},$ where $\ds{d\mu(z) =
\frac{1}{(2\pi t)^{n/2}}}$ $\ds{e^{-\frac{|y|^2}{2t}} dx dy }$ and
$\mathcal{O}(\mathbb{C}^n)$ denotes the space of entire functions
on $\mathbb{C}^n.$

This transform has attracted a lot of attention in the recent
years mainly due to the work of Hall \cite{H} where a similar
result was established for an arbitrary compact connected Lie
group $K.$ He introduced a generalization of the Segal-Bargmann
transform on a compact connected Lie group. If $K$ is such a
group, this coherent state transform maps $L^2(K)$ isometrically
onto the space of holomorphic functions in $L^2(G,\mu_t),$ where
$G$ is the complexification of $K$ and $\mu_t$ is an appropriate
positive heat kernel measure on $G.$ The generalized coherent
state transform is defined in terms of the heat kernel on the
compact group $K$ and its analytic continuation to the complex
group $G.$ Similar results have been proved for compact symmetric
spaces by Stenzel in \cite{St}.

For the Heisenberg group $\mathbb{H}^n,$ Kr$\ddot{\txt{o}}$tz et
al proved in \cite{KTX} that the image of $L^2(\mathbb{H}^n)$
under the heat kernel transform is not a weighted Bergman space
with a non-negative weight, but can be considered as a direct
integral of twisted Bergman spaces. Similar results for non
compact symmetric spaces have been proved in \cite{HM} and
\cite{KOS}.

Next, consider the following result on $\mathbb{R}$ due to Paley
and Wiener. A function $f \in L^2(\mathbb{R})$ admits a
holomorphic extension to the strip $\{ x + iy : |y| < t \}$ such
that $$\ds{ \sup_{|y| \leq s} \int_\mathbb{R} |f(x+iy)|^2 dx <
\infty ~ \forall ~ s<t}$$ if and only if \bea \label{eq}
\int_{\mathbb{R}} e^{s|\xi|} |\widetilde{f}(\xi)|^2 d\xi < \infty
~ \forall ~ s < t\eea where $\widetilde{f}$ denotes the Fourier
transform of $f.$

The condition (\ref{eq}) is the same as $$ \int_{\mathbb{R}}
|\widetilde{e^{s{\Delta}^{\frac{1}{2}}}f}(\xi)|^2 d\xi < \infty ~~
\forall ~ s < t$$ where $\Delta$ is the Laplacian on $\mathbb{R}.$
This point of view was explored by R. Goodman in Theorem 2.1 of
\cite{G1}.

The condition (\ref{eq}) also equals $$ \int_{\mathbb{R}}
|e^{i(x+iy)\xi}|^2 |\widetilde{f}(\xi)|^2 d\xi < \infty ~~ \forall
~ |y| < t.$$  Here $\xi \mapsto e^{i(x+iy)\xi}$ may be seen as the
complexification of the parameters of the unitary irreducible
representations $\xi \mapsto e^{ix\xi}$ of $\mathbb{R}.$ The above
point of view was further developed by R. Goodman in \cite{G2}
(see Theorem 3.1). Similar results were established for the
Euclidean motion group $M(2)$ of the plane $\mathbb{R}^2$ in
\cite{NS} and in the context of general motion groups
$\mathbb{R}^n \ltimes K,$ where $K$ is a compact subgroup of
$SO(n)$ in \cite{S}. Aim of this paper is to prove results
analogous to the above three, for the Heisenberg motion groups $
\mathbb{HM} = \mathbb{H}^n \ltimes K,$ where $\mathbb{H}^n$ is the
Heisenberg group and $K$ is a compact subgroup of $U(n)$ such that
$(K,\mathbb{H}^n)$ is a Gelfand pair.

The plan of this paper is as follows : In the following section
the range of the Segal-Bargmann transform on $\mathbb{HM}$ is
characterized as a direct integral of weighted Bergman spaces. The
third section is devoted to a study of Poisson integrals on
$\mathbb{HM}$ by using a Gutzmer formula for compact Lie groups
established by Lassalle in 1978 (see \cite{L}) and a Gutzmer
formula on $\mathbb{C}^{2n}.$ This is modelled after the work of
Goodman \cite{G1}. In the final section we prove the Plancherel
theorem on $\mathbb{HM}$ thereby listing the unitary, irreducible
representations on which Plancherel measure rests. Then we prove a
Paley-Wiener type theorem which characterizes functions extending
holomorphically to the complexification of $\mathbb{HM}$ which is
an analogue of Theorem 3.1 of \cite{G2}.

\vspace{0.8 in}

\section{Segal-Bargmann transform}

In this section we want to study the Segal-Bargmann transform on
the Heisenberg motion group. We recall that, for the Heisenberg
group $\mathbb{H}^n,$ it was proved by Kr$\ddot{\txt{o}}$tz et al
in \cite{KTX} that the image of $L^2(\mathbb{H}^n)$ under the heat
kernel transform is not a weighted Bergman space with a
non-negative weight, but can be considered as a direct integral of
twisted Bergman spaces. Here we prove that a similar result is
true for Heisenberg motion groups as well.

\subsection{Segal-Bargmann transform for the special
Hermite semigroup}\hspace{0.2 in}

Let $\mathbb{H}^n = \mathbb{C}^n \times \mathbb{R}$ be the
Heisenberg group with the group operation defined by $$
(z,t)\cdot(w,s) = (z+w, t+s+\frac{1}{2}Im(z \cdot \overline{w}))
~~ \txt{ where } z, w \in \mathbb{C}^n, ~ t,s \in \mathbb{R}.$$
Alternatively, we can consider $\mathbb{H}^n$ as $\mathbb{R}^n
\times \mathbb{R}^n \times \mathbb{R}$ with the group law
$$(x,y,t)(u,v,s) = (x+u,y+v, t+s+\frac{1}{2}(u \cdot y - v \cdot
x)) \txt{ for } x,y,u,v \in \mathbb{R}^n, ~ t,s \in \mathbb{R}.$$
The center of $\mathbb{H}^n$ is $\{ (0,0,t) : t \in \mathbb{R}\}.$

Let $\lambda \in \mathbb{R},$ $\lambda \neq 0.$ For suitable
functions $f$ on $\mathbb{H}^n,$ let us define a function
$f^{\lambda}$ on $\mathbb{R}^{2n}$ by $$ f^{\lambda}(x,u) =
\int_{\mathbb{R}} f(x,u,t) e^{i \lambda t} dt .$$ For $f,g \in
L^1(\mathbb{R}^{2n}),$ the $\lambda$-twisted convolution of $f$
and $g$ is defined by $$ f*_{\lambda}g(x,u) =
\int_{\mathbb{R}^{2n}} f(x',u') g(x-x',u-u')
e^{-i\frac{\lambda}{2} (u \cdot x' - x \cdot u')} dx' du'.$$ Then,
we have for Schwartz functions $f,g \in \mathcal{S}(\mathbb{H}^n)
= \mathcal{S}(\mathbb{R}^{2n+1}),$ $$ (f*g)^{\lambda} =
f^{\lambda} *_{\lambda} g^{\lambda}$$ where $*$ denotes the group
convolution on $\mathbb{H}^n.$

Let $\mathcal{L}$ denote the sublaplacian on the Heisenberg group
defined by $$ \mathcal{L} = -\sum_{j=1}^n (X_j^2 + Y_j^2) $$ where
$X_j, Y_j , j = 1, 2, \cdots , n$ together with
$\ds{T=\frac{\partial}{\partial t}}$ forms a basis for the
Heisenberg Lie algebra. For the explicit expressions for these
vector fields we refer to \cite{T}. The heat kernel for
$\mathcal{L}$ is denoted by $p_t$ and its inverse Fourier
transform in the central variable is explicitly given by (see
\cite{KTX}) $$ p_t^{\lambda}(x,u) = (4\pi)^{-n} \left(
\frac{\lambda}{\sinh \lambda t} \right)^n
e^{-\frac{\lambda}{4}\coth (\lambda t) (|x|^2 + |u|^2)}.$$ It can
be shown that $p_t^{\la}$ is the heat kernel associated to the
special Hermite operator $L_{\la}.$ That is, $e^{-tL_{\la}}f = f
*_{\la} p_t^{\la},$ for $f \in L^2(\mathbb{R}^{2n}).$ Define a
positive weight function $\mathcal{W}_t^{\lambda}$ on
$\mathbb{C}^{2n}$ by
$$ \mathcal{W}_t^{\lambda}(x+iy,u+iv) = 4^n e^{\la(u \cdot y - v
\cdot x)} p_{2t}^{\la}(2y,2v) ~~ \txt{ where }~~ x,y,u,v \in
\mathbb{R}^n.$$ Let $\mathcal{B}_t^{\la}(\mathbb{C}^{2n})$ be the
weighted Bergman space defined as $$
\mathcal{B}_t^{\la}(\mathbb{C}^{2n}) = \{ F \txt{ entire on }
\mathbb{C}^{2n} :  \int_{\mathbb{C}^{2n}} |F(z,w)|^2
\mathcal{W}_t^{\lambda}(z,w) dz dw < \infty \}.$$

\begin{thm}
The map $e^{-tL_{\la}} : f \ra f * p_t^{\la},$ called the
$\la$-twisted heat kernel transform, is a unitary operator from
$L^2(\mathbb{R}^{2n})$ to $\mathcal{B}_t^{\la}(\mathbb{C}^{2n}).$
\end{thm}

For more details and the proof see \cite{KTX}.

\subsection{Laplacian and heat kernel on Heisenberg motion groups}\hspace{0.5 in}

Let $K$ be a compact, connected Lie subgroup of
$Aut(\mathbb{H}^n),$ such that $(K,\mathbb{H}^n)$ is a Gelfand
pair. By this we mean that the convolution algebra of
$K$-invariant $L^1$-functions on $\mathbb{H}^n$ is commutative. A
maximal compact connected group of automorphisms of $\mathbb{H}^n$
is given by the unitary group $U(n)$ acting on $\mathbb{H}^n$ via
$k(z,t) = (kz,t).$ Conjugating by an automorphism of
$\mathbb{H}^n$ if necessary, we can always assume that $K \subset
U(n).$ It is well known that $(U(n),\mathbb{H}^n)$ is a Gelfand
pair and there are many proper subgroups $K$ of $U(n)$ for which
$(K,\mathbb{H}^n)$ form a Gelfand pair.

Let $\mathbb{HM}$ be the semidirect product of $\mathbb{H}^n$ and
$K$ with the group law $$(x,y,t,k)(u,v,s,h) = ((x,u,t) \cdot
(k\cdot(u,v),s),kh) \txt{ where } (x,y,t), (u,v,s) \in
\mathbb{H}^n; k, h \in K.$$ $\mathbb{HM}$ is called the Heisenberg
motion group. For $K=U(n),$ $\mathbb{HM}^n = \mathbb{H}^n \ltimes
U(n)$ is more commonly known as the Heisenberg motion group.
However, in this paper by a Heisenberg motion group $\mathbb{HM},$
we shall mean $\mathbb{H}^n \ltimes K.$ Points in $\mathbb{HM}$
will be denoted by $(x,y,t,k)$ where $(x,y,t) \in \mathbb{H}^n$
and $k \in K.$

Let $K_1, K_2, \cdots, K_N$ be an orthonormal basis of the Lie
algebra $\underline{k}$ of $K.$ In the Heisenberg motion group
$\mathbb{HM},$ we have $2n+1+N$ one parameter subgroups given by
\beas G_j &=& \{ (te_j,0,0,I) : t \in \mathbb{R} \} \\
G_{n+j} &=& \{ (0,te_j,0,I) : t \in \mathbb{R} \} \\
G_{2n+1} &=& \{ (0,0,t,I) : t \in \mathbb{R} \} \\
G_{2n+1+l} &=& \{ (0,0,0,\exp{(tK_l)}) : t \in \mathbb{R} \} \\
\eeas where $1 \leq j \leq n,$ $1 \leq l \leq N$ and $e_j$ are the
co-ordinate vectors in $\mathbb{R}^n.$ Corresponding to these one
parameter subgroups we have $2n+1+N$ left invariant vector fields
$X_1, X_2, \cdots, X_{2n+1+N},$ which form a basis of the Lie
algebra of $\mathbb{HM}.$ The Laplacian $\Delta$ on $\mathbb{HM}$
is given by $$\ds{\Delta = -(X_1^2 + X_2^2 + \cdots +
X_{2n+1+N}^2).}$$

Let $Sp(n,\mathbb{R})$ denote the symplectic group consisting of
order $2n$ real matrices with determinant one that preserve the
symplectic form $\ds{[(x,u),(y,v)] = u \cdot y - v \cdot x}.$ Let
$O(2n,\mathbb{R})$ be the orthogonal group of order $2n.$ Define
$M=Sp(n,\mathbb{R}) \bigcap O(2n,\mathbb{R}).$ Then there is a one
to one correspondence between $M$ and the unitary group $U(n).$
Let $k=a+ib$ be an $n \times n$ complex matrix with real and
imaginary parts $a$ and $b.$ Then $k$ is unitary if and only if
the matrix $\ds{\left ( \begin{matrix} a & -b \\ b & a \\
\end{matrix} \right) \in M.}$ A simple computation using the
above and the fact that $K \subset U(n)$ shows that
$$\ds{\Delta = - \Delta_{\mathbb{H}^n} - \Delta_K}$$ where
$\ds{\Delta_{\mathbb{H}^n}} = \sum_{j=1}^{2n+1} X_j^2$ and
$\ds{\Delta_{K}}$ are the Laplacians on $\mathbb{H}^n$ and $K$
respectively.

Since $\ds{\Delta_{\mathbb{H}^n}}$ and $\ds{\Delta_K}$ commute, it
follows that the heat kernel $\psi_t$ associated to $\Delta$ is
given by the product of the heat kernels $k_t$ on $\mathbb{H}^n$
and $q_t$ on $K.$ In other words \beas \psi_t(x,u,\xi,k) &=&
k_t(x,u,\xi) q_t(k) \\ &=& \left( (4\pi)^{-n} \int_{\mathbb{R}}
e^{-i \lambda \xi} e^{-t{\lambda}^2} p_t^{\lambda}(x,u)d\lambda
\right) \left( \sum_{\pi \in \widehat{K}} d_\pi
e^{-\frac{\lambda_{\pi}t}{2}} \chi_{\pi}(k) \right). \eeas Here,
for each unitary, irreducible representation $\pi$ of $K,$
$d_{\pi}$ is the degree of $\pi,$ $\lambda_{\pi}$ is such that
$\pi(\ds{\Delta_K}) = - \lambda_{\pi} I$ and $\chi_{\pi}(k) = tr
(\pi(k))$ is the character of $\pi.$ For more details, see
\cite{H}.

\subsection{Segal-Bargmann transform}\hspace{0.5in}

Denote by $G$ the complexification of $K.$ Let $\kappa_t$ be the
fundamental solution at the identity of the following equation on
$G :$ $$ \frac{du}{dt} = \frac{1}{4} \Delta_G u,$$ where
$\Delta_G$ is the Laplacian on $G,$ for details please see
\cite{H}. It should be noted that $\kappa_t$ is the real, positive
heat kernel on $G$ which is not the same as the analytic
continuation of $q_t$ on $K.$

Define $\mathcal{A}_t^{\la}(\mathbb{C}^{2n} \times G)$ to be the
weighted Bergman space $$\mathcal{A}_t^{\la}(\mathbb{C}^{2n}
\times G) = \{ F \txt{ entire on } \mathbb{C}^{2n} \times G :
\int_G \int_{\mathbb{C}^{2n}} |F(z,w,g)|^2
\mathcal{W}_t^{\lambda}(z,w) dz dw d\nu(g) < \infty \}$$ where $$
d\nu(g) = \int_K \kappa_t(xg)dx \textmd{ on } G.$$

We now introduce a measurable structure on $\ds{\bigsqcup_{\la
\neq 0} \mathcal{A}_t^{\la}(\mathbb{C}^{2n} \times G)}.$ By a
section $s$ of $\ds{\bigsqcup_{\la \neq 0}
\mathcal{A}_t^{\la}(\mathbb{C}^{2n} \times G)}$ we mean an
assignment \beas s : \mathbb{R}^* &\ra& \ds{\bigsqcup_{\la \neq 0}
\mathcal{A}_t^{\la}(\mathbb{C}^{2n} \times G)} \\
\la &\ra& s_{\la} \in \mathcal{A}_t^{\la}(\mathbb{C}^{2n} \times
G).\eeas Now we define a direct integral of Hilbert spaces by $$
\int_{\mathbb{R}^*}^{\oplus} \mathcal{A}_t^{\la}(\mathbb{C}^{2n}
\times G) e^{2t\la^2} d\la = \left\{ s : \mathbb{R}^* \ra
\ds{\bigsqcup_{\la \neq 0} \mathcal{A}_t^{\la}(\mathbb{C}^{2n}
\times G)} \txt{ such that } s \txt{ is measurable and } \right.$$
$$ \left. \|s\|^2 = \int_{\mathbb{R}^*} \|s_{\la}\|_{\la}^2 e^{2t\la^2}
d\la < \infty \right\}$$ where $\| \cdot \|_{\la}$ denotes the
norm in $\ds{\mathcal{A}_t^{\la}(\mathbb{C}^{2n} \times G)}.$
Clearly this is a Hilbert space.

For suitable functions $f$ on $\mathbb{HM},$ let us define a
function $f^{\lambda}$ on $\mathbb{R}^{2n} \times K$ by $$
f^{\lambda}(x,u,k) = \int_{\mathbb{R}} f(x,u,t,k) e^{i \lambda t}
dt.$$ Notice that this definition is consistent with the one used
earlier for functions on $\mathbb{H}^n$ (i.e. for right
$K$-invariant functions on $\mathbb{HM}$). Then we have the
following theorem :

\begin{thm}\label{sbt}
If $f \in L^2(\mathbb{HM}),$ then $f*\psi_t$ extends
holomorphically to $\mathbb{C}^{2n+1} \times G.$
\begin{itemize} \item [(a)] The image of $L^2(\mathbb{HM})$ under
the Segal-Bargmann transform $f \ra f * \psi_t$ cannot be
characterized as a weighted Bergman space with a non-negative
weight. \item [(b)] For every $t>0,$ the Segal-Bargmann transform
$\ds{e^{-t\Delta} : L^2(\mathbb{HM}) \ra }$ $\ds{\bigsqcup_{\la
\neq 0} \mathcal{A}_t^{\la}(\mathbb{C}^{2n}}$ $\ds{\times G),}$ $f
\ra (f*\psi_t)^{\la}$ is an isometric isomorphism.\end{itemize}
\end{thm}

\begin{proof}

Let $f \in L^2(\mathbb{HM}).$ Expanding $f$ in the $K$-variable
using the Peter-Weyl theorem we obtain $$ f(x,u,t,k) = \sum_{\pi
\in \widehat{K}} d_{\pi} \sum_{i,j = 1}^{d_\pi}
f_{ij}^{\pi}(x,u,t) \phi_{ij}^{\pi}(k)$$ where for each $\pi \in
\widehat{K},$ $d_\pi$ is the degree of $\pi,$ $\phi_{ij}^{\pi}$'s
are the matrix coefficients of $\pi$ and $\ds{f_{ij}^{\pi}(x,u,t)
= \int_K f(x,u,t,k) \overline{\phi_{ij}^{\pi}(k)} dk}.$ Here, the
convergence is understood in the $L^2$-sense. Moreover, by the
universal property of the complexification of a compact Lie group
(see Section 3 of \cite{H}), all the representations of $K,$ and
hence all the matrix entries, extend holomorphically to $G.$

Since $\psi_t$ is $K$-invariant (as a function on $\mathbb{H}^n$)
a simple computation shows that $$ f * \psi_t (x,u,t,k) =
\sum_{\pi \in \widehat{K}} d_{\pi} e^{-\frac{\lambda_{\pi}t}{2}}
\sum_{i,j = 1}^{d_\pi} f_{ij}^{\pi} * k_t(x,u,t)
\phi_{ij}^{\pi}(k),$$ where the convolution on the right is the
one on $\mathbb{H}^n.$ It is easily seen that $f_{ij}^{\pi} \in
L^2(\mathbb{H}^n)$ for every $\pi \in \widehat{K}$ and $1 \leq i,j
\leq d_{\pi}.$ Hence $f_{ij}^{\pi} * k_t$ extends to a holomorphic
function on $\mathbb{C}^{2n+1}.$ We formally define the analytic
continuation of $f * \psi_t$ to $\mathbb{C}^{2n+1} \times G$ by $$
f * \psi_t (z,w,\zeta,g) = \sum_{\pi \in \widehat{K}} d_{\pi}
e^{-\frac{\lambda_{\pi}t}{2}} \sum_{i,j = 1}^{d_\pi} f_{ij}^{\pi}
* p_t(z,w,\zeta) \phi_{ij}^{\pi}(g)$$ where $(z,w,\zeta) \in
\mathbb{C}^{2n+1}$ and $g \in G.$

We claim that the above series converges uniformly on compact
subsets of $\mathbb{C}^{2n+1} \times G$ so that $f * \psi_t$
extends to an entire function on $\mathbb{C}^{2n+1} \times G.$ We
know from Section 4, Proposition 1 of \cite{H} that the
holomorphic extension of the heat kernel $q_t$ on $K$ is given by
$$q_t(g) = \sum_{\pi \in \widehat{K}} d_\pi e^{-\frac{\lambda_{\pi}t}{2}}
\chi_{\pi}(g).$$ For each $g \in G,$ define the function $q_t^g(k)
= q_t(gk).$ Then $q_t^g$ is a smooth function on $K$ and is given
by \beas q_t^g(k) &=& \sum_{\pi \in \widehat{K}} d_\pi
e^{-\frac{\lambda_{\pi}t}{2}} \chi_{\pi}(gk)\\
&=& \sum_{\pi \in \widehat{K}} d_\pi e^{-\frac{\lambda_{\pi}t}{2}}
\sum_{i,j=1}^{d_{\pi}} \phi_{ij}^{\pi}(g) \phi_{ji}^{\pi}(k).
\eeas Since $q_t^g$ is a smooth function on $K,$ we have for each
$g \in G,$ \bea \label{e1} \int_K |q_t^g(k)|^2 dk = \sum_{\pi \in
\widehat{K}} d_\pi e^{-\lambda_{\pi}t} \sum_{i,j=1}^{d_{\pi}}
|\phi_{ij}^{\pi}(g)|^2 < \infty .\eea Let $L$ be a compact set in
$\mathbb{C}^{2n+1} \times G.$ For $(z,w,\zeta,g) \in L$ we have,
\bea \label{e2} |f * \psi_t (z,w,\zeta, g)| \leq \sum_{\pi \in
\widehat{K}} d_{\pi} e^{-\frac{\lambda_{\pi}t}{2}} \sum_{i,j =
1}^{d_\pi} |f_{ij}^{\pi} * p_t(z,w,\zeta)| |\phi_{ij}^{\pi}(g)|.
\eea It is known that the inclusion
$\ds{e^{-t\Delta_{\mathbb{H}^n}} L^2(\mathbb{H}^n) \hookrightarrow
\mathcal{O}(\mathbb{C}^{2n+1})}$ is continuous (see Section 3 of
\cite{KTX}) i.e. there exists a constant $C_L$ depending on $L$
such that for any $h \in L^2(\mathbb{H}^n),$ $$ \sup_{(z,w,\zeta)
\in L'} |h*k_t(z,w,\zeta)| \leq C_L \|h\|_{L^2(\mathbb{H}^n)}$$
where $L'$ is the projection of $L$ to $\mathbb{C}^{2n+1}.$ Using
the above in (\ref{e2}) and applying Cauchy-Schwarz inequality we
get \beas |f
* \psi_t (z,w,\zeta,g)| &\leq& C_L \sum_{\pi \in \widehat{K}}
d_{\pi} \sum_{i,j = 1}^{d_\pi}
\|f_{ij}^{\pi}\|_{L^2(\mathbb{H}^n)}
e^{-\frac{\lambda_{\pi}t}{2}} |\phi_{ij}^{\pi}(g)| \\
&\leq& C_L \left( \sum_{\pi \in \widehat{K}} d_{\pi} \sum_{i,j =
1}^{d_\pi} \int_{\mathbb{H}^n} |f_{ij}^{\pi}(X)|^2 dX \right)
^{\frac{1}{2}} \left( \sum_{\pi \in \widehat{K}} d_{\pi} \sum_{i,j
= 1}^{d_\pi} e^{-\lambda_{\pi}t} |\phi_{ij}^{\pi}(g)|^2
\right)^{\frac{1}{2}} .\eeas Noting that $\ds{ \|f\|_2^2 =
\sum_{\pi \in \widehat{K}} d_{\pi} \sum_{i,j = 1}^{d_\pi}
\int_{\mathbb{H}^n} |f_{ij}^{\pi}(X)|^2 dX }$ and $q_t$ is a
smooth function on $G$ we prove the claim using (\ref{e1}). Hence
$f*\psi_t$ extends holomorphically to $\mathbb{C}^{2n+1} \times
G.$

Next, we want to prove the non-existence of a non-negative weight
function $W_t$ on $\mathbb{C}^{2n+1} \times G$ such that for every
$f \in L^2(\mathbb{HM}),$ we have
$$\int_K \int_{\mathbb{H}^n} |f(x,u,t,k)|^2 dx du dt dk = \int_{\mathbb{C}^{2n+1}
\times G} |f*\psi_t(z,w,\zeta,g)|^2 W_t(z,w,\zeta,g) dz dw d\zeta
dg.$$ If we take $f$ such that $f(x,u,t,k)=h(x,u,t)$ for $h \in
L^2(\mathbb{H}^n),$ then we get from the above relation that
$$ \int_{\mathbb{H}^n} |h(x,u,t)|^2 dx du dt =
\int_{\mathbb{C}^{2n+1}} |h*k_t(z,w,\zeta)|^2 \left( \int_G
W_t(z,w,\zeta,g) dg \right)dz dw d\zeta.$$ Since $W_t$ is assumed
to be non-negative, this clearly gives a contradiction to the fact
that the image of $L^2(\mathbb{H}^n)$ under the Segal-Bargmann
transform $h \ra h * k_t$ cannot be characterized as a weighted
Bergman space with a non-negative weight (see \cite{KTX} for
details).

For any $f \in L^2(\mathbb{HM}),$ define $f_X(k)= f(X,k)$ for $X
\in \mathbb{H}^n$ and $k \in K.$ Then using Theorem 2 in \cite{H}
we have \beas \|f\|_2^2 &=&
\int_{\mathbb{H}^n} \int_K |f(X,k)|^2 dk dX \\
&=& \int_{\mathbb{H}^n} \int_G |f_X * q_t(g)|^2 d\nu(g) dX.\eeas
Define $F_g(X) = f_X * q_t(g)$ for $g \in G.$ Then, using Theorem
5.1 of \cite{KTX} we get that \beas \|f\|_2^2 &=& \int_G
\int_{\mathbb{H}^n} |F_g(X)|^2 dX d\nu(g) \\
&=& \int_G \int_{\mathbb{R}^*} e^{2t\la^2} \|(F_g
*k_t)^{\la}\|^2_{\mathcal{B}_t^{\la}(\mathbb{C}^{2n})} d\la d\nu(g)\\
&=& \int_G \int_{\mathbb{R}^*} e^{2t\la^2}
\int_{\mathbb{C}^{2n}}|(F_g *k_t)^{\la}(z,w)|^2
\mathcal{W}_t^{\lambda}(z,w) dz dw d\la d\nu(g). \eeas It is easy
to see that $F_g *k_t(x,u)=f*\psi_t(x,u,g).$ Since the functions
on both the sides extend holomorphically to $\mathbb{C}^{2n},$ we
have $F_g *k_t(z,w)=f*\psi_t(z,w,g)$ for every $z,w \in
\mathbb{C}^n$ and $g \in G.$ Hence it follows that \beas \|f\|_2^2
&=& \int_{\mathbb{R}^*} \int_G e^{2t\la^2}
\int_{\mathbb{C}^{2n}}|(f*\psi_t)^{\la}(z,w,g)|^2
\mathcal{W}_t^{\lambda}(z,w) dz dw d\nu(g) d\la. \eeas

It remains to prove the surjectivity part. Let $\ds{s \in
\int_{\mathbb{R}^*}^{\oplus} \mathcal{A}_t^{\la}(\mathbb{C}^{2n}
\times G) e^{2t\la^2} d\la}.$ Then $s_{\la} \in
\mathcal{A}_t^{\la}(\mathbb{C}^{2n} \times G)$ and
$\ds{\int_{\mathbb{R}^*} \|s_{\la}\|_{\la}^2 e^{2t\la^2} d\la <
\infty.}$ Now $e^{-t\Delta_K} e^{-tL_{\la}}$ is a unitary map from
$L^2(\mathbb{R}^{2n} \times K)$ onto
$\mathcal{A}_t^{\la}(\mathbb{C}^{2n} \times G).$ So there exists
$g_{\la} \in L^2(\mathbb{R}^{2n} \times K)$ for each $\la \neq 0$
such that $e^{-t\Delta_K} e^{-tL_{\la}}g_{\la} = e^{t\la^2}
s_{\la}$ and $\ds{\int_{\mathbb{R}^*}
\|g_{\la}\|_{L^2(\mathbb{R}^{2n} \times K)}^2 d\la < \infty.}$
This implies that there exists a unique $f \in L^2(\mathbb{HM})$
such that $f^{\la}(x,u) = g_{\la}(x,u)$ almost everywhere. Finally
we have $(f*\psi_t)^{\la} = e^{-t\la^2} e^{-t\Delta_K}
e^{-tL_{\la}}g_{\la}$ in $L^2.$ Using the above equalities we have
$(f*\psi_t)^{\la} = s_{\la}.$ This proves the surjectivity and
hence the theorem.

\end{proof}

\vspace{0.8 in}

\section{Gutzmer's formula and Poisson Integrals}

In this section, we briefly recall Gutzmer's formula on compact,
connected Lie groups given by Lassalle in \cite{L}. Then we prove
a Gutzmer type formula for functions on $\mathbb{C}^{2n}$ with
respect to the $K$-action. With the help of the above Gutzmer's
formulae, we characterize Poisson integrals on the Heisenberg
motion groups. We also give a generalization of the Segal-Bargmann
transform on Heisenberg motion groups.

\subsection{Gutzmer's formula on compact, connected Lie
groups}\hspace{0.2in}

Let $\underline{k}$ and $\underline{g}$ be the Lie algebras of a
compact, connected Lie group $K$ and its complexification $G.$
Then we can write $\underline{g} = \underline{k} + \underline{p}$
where $\underline{p}=i\underline{k}$ and any element $g \in G$ can
be written in the form $g=k\exp{iH}$ for some $k \in K, ~ H \in
\underline{k}.$ If $\underline{h}$ is a maximal, abelian
subalgebra of $\underline{k}$ and $\underline{a}=i\underline{h}$
then every element of $\underline{p}$ is conjugate under $K$ to an
element of $\underline{a}.$ Thus each $g \in G$ can be written
(non-uniquely) in the form $g = k_1 \exp {(iH)} k_2 $ for $k_1,
k_2 \in K$ and $H \in \underline{h}.$ If $ k_1 \exp {(iH_1)} k_1'
= k_2 \exp {(iH_2)} k_2', $ then there exists $w \in W,$ the Weyl
group with respect to $\underline{h},$ such that $H_1 = w \cdot
H_2$ where $\cdot$ denotes the action of the Weyl group on
$\underline{h}.$ Since $K$ is compact, there exists an Ad$-
K$-invariant inner product on $\underline{k},$ and hence on
$\underline{g}.$ Let $|\cdot|$ denote the norm with respect to the
said inner product. Then we have the following Gutzmer's formula
by Lassalle (see \cite{L}).

\begin{thm}\label{lassalle}
Let $f$ be a holomorphic function in $K\exp(i\Omega_r)K \subseteq
G$ where $\Omega_r = \{ H \in \underline{k} : |H| <r \}.$ Then we
have $$ \int_K \int_K |f(k_1 \exp{iH} k_2)|^2 dk_1 dk_2 =
\sum_{\pi \in \widehat{K}} \|\widehat{f}(\pi)\|_{HS}^2
\chi_{\pi}(\exp{2iH})$$ where $H \in \Omega_r$ and
$\widehat{f}(\pi)$ is the operator-valued Fourier transform of $f$
at $\pi$ defined by $\ds{\widehat{f}(\pi) = \int_K f(k)
\pi(k^{-1}) dk.}$
\end{thm}

For the proof of above, see \cite{L}.

\subsection{The Hermite and special Hermite
functions}\hspace{0.1in}\label{hermite}

Here we collect relevant information about Hermite and special
Hermite functions. We closely follow the notation used in \cite{T}
and we refer the reader to the same for more details.

For every $\la \neq 0,$ the Schr$\ddot{\txt{o}}$dinger
representation $\pi_{\la}$ of the Heisenberg group $\mathbb{H}^n$
on $L^2(\mathbb{R}^n)$ is defined by
$$\pi_{\la}(x,u,t)f(\xi) = e^{i\la t} e^{i\la (x \cdot \xi +
\frac{1}{2} x \cdot u)} f(\xi + u)$$ where $f \in
L^2(\mathbb{R}^n).$ A celebrated theorem of Stone-von Neumann says
that up to unitary equivalence these are all the irreducible
unitary representations of $\mathbb{H}^n$ that are nontrivial at
the center.

\begin{thm}
The representations $\pi_{\la},$ $\la \neq 0$ are unitary and
irreducible. If $\rho$ is an irreducible unitary representation of
$\mathbb{H}^n$ on a Hilbert space $\mathcal{H}$ such that
$\rho(0,0,t)= e^{i\la t}I$ for some $\la \neq 0,$ then $\rho$ is
unitarily equivalent to $\pi_{\la}.$
\end{thm}

Note that $\ds{\pi_{\la}(x,u,t) = e^{i \la t} \pi_{\la}(x,u,0)}.$
We shall write $\pi_{\la}(x,u,0)$ as $\pi_{\la}(x,u).$ Let
$\phi_{\alpha},$ $\alpha \in \mathbb{N}^n$ be the Hermite
functions on $\mathbb{R}^n$ normalized so that their $L^2$ norms
are one. The family $\{\phi_{\alpha},$ $\alpha \in \mathbb{N}^n\}$
is an orthonormal basis of $L^2(\mathbb{R}^n).$ For $\la \neq 0,$
we define the scaled Hermite functions $$ \phi_{\alpha}^{\la}(x) =
|\la|^{\frac{n}{4}} \phi_{\alpha}(|\la|^{\frac{1}{2}} x).$$ We
also consider $$ \phi_{\alpha \beta}^{\la}(x,u) =
(2\pi)^{-\frac{n}{2}} |\la|^{\frac{n}{2}} \langle \pi_{\la}(x,u,0)
\phi_{\alpha}^{\la}, \phi_{\beta}^{\la}\rangle, ~~ \txt{ for }
\alpha, \beta \in \mathbb{N}^n$$ which are essentially the matrix
coefficients of $\pi_{\la}$ at $(x,u,0) \in \mathbb{H}^n.$ They
are the so called special Hermite functions and  $\{\phi_{\alpha
\beta}^{\la} : \alpha, \beta \in \mathbb{N}^n\}$ is a complete
orthonormal system in $L^2(\mathbb{R}^{2n}).$

Note that $\phi_{\alpha}^{\la}(\xi) = H_{\alpha}^{\la}(\xi)
e^{-\frac{\la}{2} |\xi|^2}$ where $H_{\alpha}^{\la}$ is a
polynomial on $\mathbb{R}^n.$ For $z \in \mathbb{C}^n,$ we define
$\phi_{\alpha}^{\la}(z)$ to be $H_{\alpha}^{\la}(z)
e^{-\frac{\la}{2} z^2}$ where $z^2 = z \cdot z.$ Then for $z,w \in
\mathbb{C}^n$ we can define
$$\pi_{\la}(z,w,t)\phi_{\alpha}^{\la}(\xi) = e^{i\la t} e^{i\la (z \cdot
\xi + \frac{1}{2} z \cdot w)} \phi_{\alpha}^{\la}(\xi + w).$$
Hence we have $$ \phi_{\alpha \beta}^{\la}(z,w) =
(2\pi)^{-\frac{n}{2}} |\la|^{\frac{n}{2}} \langle \pi_{\la}(z,w)
\phi_{\alpha}^{\la}, \phi_{\beta}^{\la}\rangle$$ for $z,w \in
\mathbb{C}^n.$ An easy calculation shows that \bea\label{schr}
\langle \pi_{\la}(z,w) \phi_{\alpha}^{\la},
\phi_{\beta}^{\la}\rangle = \langle \phi_{\alpha}^{\la},
\pi_{\la}(-\bar{z},-\bar{w}) \phi_{\beta}^{\la}\rangle.\eea Notice
that both $\phi_{\alpha}^{\la}(z),$ $\phi_{\alpha
\beta}^{\la}(z,w)$ are holomorphic on $\mathbb{C}^n$ and
$\mathbb{C}^{2n}$ respectively.

\subsection{Gelfand pairs and the Heisenberg group}\hspace{0.1in}
\label{s1}

For each $k \in K \subseteq U(n),$ $(x,u,t) \ra (k \cdot (x,u),t)$
is an automorphism of $\mathbb{H}^n,$ because $U(n)$ preserves the
symplectic form $x \cdot u - y \cdot v.$ If $\rho$ is a
representation of $\mathbb{H}^n,$ then using this automorphism we
can define another representation $\rho_k$ by $\rho_k(x,u,t) =
\rho(k \cdot (x,u),t)$ which coincides with $\rho$ at the center.
If we take $\rho$ to be the Schr$\ddot{\txt{o}}$dinger
representation $\pi_{\la},$ then by Stone-von Neumann theorem
$(\pi_{\la})_k$ is unitarily equivalent to $\pi_{\la}$ and we have
the unitary intertwining operator $\mu_{\la}$ such that \bea
\label{meta} \pi_{\la}(k \cdot (x,u),t) = \mu_{\la}(k)
\pi_{\la}(x,u,t) \mu_{\la}(k)^*.\eea

The operator valued function $\mu_{\la}$ can be chosen so that it
becomes a unitary representation of $K$ on $L^2(\mathbb{R}^n)$ and
is called the metaplectic representation. For each $m>0,$ let
$\mathcal{P}_m$ be the linear span of $\{ \phi_{\alpha} : |\alpha|
= m \}.$ Then each such $\mathcal{P}_m$ is invariant under the
action of $\mu_{\la}(k)$ for every $k \in K \subseteq U(n).$ When
$K=U(n),$ $\mu_{\la}|_{\mathcal{P}_m}$ is irreducible. When $K$ is
a proper compact subgroup of $U(n),$ $\mathcal{P}_m$ need not be
irreducible under the action of $\mu_{\la}.$ So it further
decomposes into irreducible subspaces. It is known that
$(K,\mathbb{H}^n)$ is a Gelfand pair if and only if this action of
$K$ on $L^2(\mathbb{R}^n)$ decomposes into irreducible components
of multiplicity one (see \cite{BJR}).

Let $L_m^{n-1}$ be the Laguerre polynomials of type $(n-1)$ and
define Laguerre functions by $$\varphi_m^{\la}(x,u) =
L_m^{n-1}\left(\frac{|\la|}{2}(|x|^2 + |u|^2)\right)
e^{-\frac{|\la|}{4}(|x|^2+|u|^2)}.$$ Then it is known that
$$\varphi_m^{\la}(x,u)=\sum_{|\alpha|=m} \phi_{\alpha
\alpha}^{\la}(x,u)$$ and $\ds{e_m^{\la}(x,u,t) = }\frac{1}{\dim
\mathcal{P}_m} e^{i \la t} \varphi_m^{\la}(x,u)$ is a
$U(n)$-spherical function. Let $\ds{\mathcal{P}_m =
\bigoplus_{a=1}^{A_m} \mathcal{P}_{ma}}$ be the decomposition of
$\mathcal{P}_m$ into $K$-irreducible subspaces. Then
$$e_{ma}^{\la}(x,u,t) = \frac{1}{\dim
\mathcal{P}_{ma}} e^{i \la t} \varphi_{ma}^{\la}(x,u) =
\frac{1}{\dim \mathcal{P}_{ma}} \sum_{b=1}^{B_a} \langle
\pi_{\la}(x,u,t) \phi_{ma}^b , \phi_{ma}^b \rangle $$ is a
$K$-spherical function where $\{ \phi_{ma}^b : b = 1 \cdots B_a
\}$ is an orthonormal basis for $\mathcal{P}_{ma}$ such that $\{
\phi_{ma}^b : b = 1 \cdots B_a, a = 1 \cdots A_m \}$ is an
orthonormal basis for $\mathcal{P}_m.$ For more on Gelfand pairs
and spherical functions on $\mathbb{H}^n$ see \cite{BJR}.

\subsection{Gutzmer's formula for $K$-special Hermite
functions}\hspace{0.1in}\label{kfns}

Let us write $\{\phi_{ma}^b : b = 1 \cdots B_a, a = 1 \cdots
A_m\}$ as $\{\psi_{\alpha} : \alpha \in \mathbb{N}^n \}$ such that
for each $m,$ $\{\phi_{ma}^b : b = 1 \cdots B_a, a = 1 \cdots
A_m\}$ are the ones which occur as $\psi_{\alpha}$ for
$|\alpha|=m.$ For $\la \neq 0,$ we define
$$ \psi_{\alpha}^{\la}(x) = |\la|^{\frac{n}{4}}
\psi_{\alpha}(|\la|^{\frac{1}{2}} x).$$ Consider $$ \psi_{\alpha
\beta}^{\la}(x,u) = (2\pi)^{-\frac{n}{2}} |\la|^{\frac{n}{2}}
\langle \pi_{\la}(x,u) \psi_{\alpha}^{\la},
\psi_{\beta}^{\la}\rangle$$ and we call them $K$-special Hermite
functions. It is easy to see that $\{\psi_{\alpha \beta}^{\la} :
\alpha, \beta \in \mathbb{N}^n\}$ is a complete orthonormal system
in $L^2(\mathbb{R}^{2n}).$ Since each $\psi_{\alpha}$ is a finite
linear combination of $\phi_{\alpha}$'s, both
$\psi_{\alpha}^{\la}$ and $\psi_{\alpha \beta}^{\la}$ extend as
holomorphic functions to $\mathbb{C}^n$ and $\mathbb{C}^{2n}$
respectively for each $\alpha, \beta \in \mathbb{N}^n.$ We also
note that the action of $K \subseteq U(n)$ on $\mathbb{R}^{2n}$
naturally extends to an action of $G$ on $\mathbb{C}^{2n}.$

\begin{thm}\label{Gutzmer}
For a function $F \in L^2(\mathbb{R}^{2n})$ having a holomorphic
extension to $\mathbb{C}^{2n},$ we have \beas && \int_K
\int_{\mathbb{R}^{2n}} |F(k \cdot (x+iy,u+iv))|^2 e^{\la(u \cdot y
- v \cdot x)} dx du dk \\ &=& \sum_{m=0}^{\infty} \sum_{a=1}^{A_m}
(\dim \mathcal{P}_{ma})^{-1} \varphi_{ma}^{\la}(2iy,2iv) \| F
*_{\la} \varphi_{ma}^{\la} \|^2, \eeas whenever either of them is
finite.
\end{thm}

\begin{proof}

First we want to prove that $\psi_{\alpha \beta}^{\la}$'s are
orthogonal under the inner product $$ \langle F, G \rangle =
\int_K \int_{\mathbb{R}^{2n}} F(k \cdot (x+iy,u+iv)) ~
\overline{G(k \cdot (x+iy,u+iv))} ~ e^{\la(u \cdot y - v \cdot x)}
~ dx ~ du ~ dk $$ for $F, ~ G \in L^2(\mathbb{R}^n)$ which have a
holomorphic extension to $\mathbb{C}^{2n}.$ So, we consider \beas
&& \int_K \int_{\mathbb{R}^{2n}} \psi_{\alpha \beta}^{\la}(k \cdot
(x+iy,u+iv)) ~ \overline{\psi_{\mu \nu}^{\la}(k \cdot
(x+iy,u+iv))} ~ e^{\la(u \cdot y - v \cdot x)} ~ dx ~ du ~ dk \\
&=& (2\pi)^{-n} |\la|^{n} \int_K \int_{\mathbb{R}^{2n}} \left
\langle \pi_{\la}(k \cdot(x,u)) \psi_{\alpha}^{\la}, \pi_{\la}(k
\cdot(iy,iv)) \psi_{\beta}^{\la} \right \rangle ~ \overline{\left
\langle \pi_{\la}(k \cdot(x,u)) \psi_{\mu}^{\la}, \right.}  \\ &&
\overline{ \left. \pi_{\la}(k \cdot(iy,iv)) \psi_{\nu}^{\la}
\right \rangle} ~ dx ~ du ~ dk \eeas by (\ref{schr}). Expanding
$\pi_{\la}(k \cdot(x,u)) \psi_{\alpha}^{\la}$ in terms of
$\psi_{\rho}^{\la},$ $\pi_{\la}(k \cdot(x,u)) \psi_{\mu}^{\la}$ in
terms of $\psi_{\sigma}^{\la}$ and using the self-adjointness of
$\pi_{\la}(k \cdot(iy,iv))$ the above equals \beas && \sum_{\rho,
\sigma \in \mathbb{N}^n} \int_K \left \langle \pi_{\la}(k
\cdot(iy,iv)) \psi_{\rho}^{\la}, \psi_{\beta}^{\la} \right \rangle
~ \overline{\left \langle \pi_{\la}(k \cdot(iy,iv))
\psi_{\sigma}^{\la}, \psi_{\nu}^{\la} \right \rangle} \left(
\int_{\mathbb{R}^{2n}} \psi_{\alpha \rho}^{\la} (k \cdot (x,u))
\right.\\ && \left. \overline{\psi_{\mu \sigma}^{\la} (k \cdot
(x,u))}~ dx ~ du \right ) ~ dk \\
&=& \delta_{\alpha,\mu} \int_K \left \langle \pi_{\la}(k
\cdot(2iy,2iv)) \psi_{\nu}^{\la}, \psi_{\beta}^{\la} \right
\rangle ~ dk, \eeas $\delta$ being the Kronecker delta. Then using
(\ref{meta}), expanding $\mu_{\la}(k^{-1}) \psi_{\nu}^{\la}$ in
terms of $\psi_{\gamma}^{\la},$ for $\gamma \in \mathcal{P}_{ma},$
$\mu_{\la}(k^{-1}) \psi_{\beta}^{\la}$ in terms of
$\psi_{\delta}^{\la}$ for $\delta \in \mathcal{P}_{lb}$ and using
Schur's orthogonality relations we get that the above equals \beas
&& \delta_{\alpha,\mu} \int_K \left \langle \pi_{\la}(2iy,2iv)
\mu_{\la}(k^{-1}) \psi_{\nu}^{\la},
\mu_{\la}(k^{-1}) \psi_{\beta}^{\la} \right \rangle ~ dk \\
&=& \delta_{\alpha,\mu} \sum_{\gamma \in \mathcal{P}_{ma}}
\sum_{\delta \in \mathcal{P}_{lb}} \left( \int_K \eta_{\gamma
\nu}(k^{-1}) \overline{\eta_{\delta \beta}(k^{-1})} dk \right)
\left \langle \pi_{\la}(2iy,2iv) \psi_{\gamma}^{\la}, \psi_{\delta}^{\la}
\right \rangle \\
&=& \delta_{\alpha,\mu} ~ \delta_{\beta,\nu} ~ \dim
\mathcal{P}_{ma}^{-1} ~ \varphi_{ma}^{\la}(2iy,2iv) \eeas where
$\mathcal{P}_{ma}$ and $\mathcal{P}_{lb}$ are the ones which
contain $\psi_{\nu}^{\la}$ and $\psi_{\beta}^{\la}$ respectively
and $\eta_{\gamma \nu}$'s are the matrix coefficients of
$\mu_{\la}.$

Then for a function $F$ as in the statement of the theorem we have
\beas && \int_K \int_{\mathbb{R}^{2n}} |F(k \cdot (x+iy,u+iv))|^2
e^{\la(u \cdot y - v \cdot x)} dx du dk \\ &=& \sum_{m=0}^{\infty}
\sum_{a=1}^{A_m} (\dim \mathcal{P}_{ma})^{-1}
\varphi_{ma}^{\la}(2iy,2iv) \left( \sum_{\alpha \in \mathbb{N}^n}
\sum_{\beta \in \mathcal{P}_{ma}} |\langle F , \phi_{\alpha
\beta}^{\la} \rangle|^2 \right).\eeas Now, it is easy to see from
standard arguments (see \cite{T} for details) that $\ds{ \| F
*_{\la} \varphi_{ma}^{\la} \|^2}$ $\ds{= \sum_{\alpha \in
\mathbb{N}^n} \sum_{\beta \in \mathcal{P}_{ma}} |\langle F ,
\psi_{\alpha \beta}^{\la} \rangle|^2} $ and hence the theorem
follows.

\end{proof}

\subsection{Poisson integrals for the Laplacian on Heisenberg motion
groups} \hspace{0.1in}

Proposition 4.1 of \cite{SR} gives that for $f \in
L^2(\mathbb{H}^n),$ $$ f(x,u,t) = (2\pi)^{-n} \sum_{m=0}^{\infty}
\sum_{a=1}^{A_m} \int_{\mathbb{R}} f * e_{ma}^{\la}(x,u,t) |\la|^n
d\la. $$ Recall that the Laplacian $\Delta$ on $\mathbb{HM}$ is
given by $\ds{\Delta = -\Delta_{\mathbb{H}^n} - \Delta_K}$ and for
suitable functions $f$ on $\mathbb{HM},$ a function $f^{\lambda}$
is defined on $\mathbb{R}^{2n} \times K$ by
$$ f^{\lambda}(x,u,k) = \int_{\mathbb{R}} f(x,u,t,k) e^{i \lambda
t} dt.$$ For $f \in L^2(\mathbb{HM}),$ we have the expansion
$$ f(x,u,t,k) = \sum_{\pi \in \widehat{K}} d_{\pi} \sum_{i,j =
1}^{d_\pi} f_{ij}^{\pi}(x,u,t) \phi_{ij}^{\pi}(k)$$ where
$\ds{f_{ij}^{\pi}(x,u,t) = \int_K f(x,u,t,k)
\overline{\phi_{ij}^{\pi}(k)} dk}.$ Then it is easy to see that
\beas && e^{-q\Delta^{\frac{1}{2}}} f(x,u,t,k) \\ &=& (2\pi)^{-n}
\sum_{\pi \in \widehat{K}} d_{\pi} \sum_{i,j = 1}^{d_\pi}
\left(\int_{\mathbb{R}} \sum_{m=0}^{\infty} \sum_{a=1}^{A_m}
e^{-q((2m+n)|\la| + \la^2 + \lambda_{\pi})^{\frac{1}{2}}}
(f_{ij}^{\pi})^{\la} *_{\la} \varphi_{ma}^{\la}(x,u) e^{i \la
t}|\la|^n d\la \right) \\ && \phi_{ij}^{\pi}(k).\eeas We have the
following (almost) characterization of the Poisson integrals. Let
$\ds{\Omega_{p,p'}}$ be the domain in $\ds{\mathbb{C}^{2n+1}
\times G}$ defined by $\ds{\{ (z,w,\tau,g) \in \mathbb{C}^{n}
\times \mathbb{C}^n \times \mathbb{C} \times G : |Im (z,w)|<p,}$
$\ds{|H|<p' \txt{ where } g = k_1 (\exp{iH}) k_2, k_1, k_2}$
$\ds{\in K, H \in \underline{h} \}}.$ Notice that the domain
$\Omega_{p,p'}$ is well defined since $|\cdot|$ is invariant under
the Weyl group action.

\begin{thm}\label{poisson}
Let $f \in L^1 \bigcap L^2(\mathbb{HM})$ be such that $f^{\la}$ is
compactly supported as a function of $\la.$ Then there exists a
constant $N$ such that for each $0<p<q,$ $h =
e^{-q\Delta^{\frac{1}{2}}} f$ extends to a holomorphic function on
the domain $\Omega_{\frac{p}{2},\frac{p}{N\sqrt{2}}}$ and \beas &&
\int_K \int_K \int_{|Im (z,w)|=r} \int_{\mathbb{H}^n}
|h(X \cdot (z,w,\tau,k_1\exp{(iH)}k_2))|^2 dX d\mu_r dk_1 dk_2 \\
&=& \sum_{\pi \in \widehat{K}} d_{\pi} \chi_{\pi}(\exp 2iH)
\sum_{m=0}^{\infty} \sum_{a=1}^{A_m} (\dim \mathcal{P}_{m})^{-1}
\int_{\mathbb{R}} L_m^{n-1}(-2\la r^2) e^{\la r^2} e^{2 \la Im \tau} \\
&& e^{-2q((2m+n)|\la| + \la^2 + \lambda_{\pi})^{\frac{1}{2}}}
\sum_{i,j = 1}^{d_\pi} \| (f_{ij}^{\pi})^{\la} *_{\la}
\varphi_{ma}^{\la} \|^2 d\la \eeas where $\mu_r$ is the normalized
surface area measure on the sphere $\{ |Im (z,w)| = r \} \subseteq
\mathbb{R}^{2n}$ for $r<\frac{p}{2}$ and $L_m^{n-1}$ are the
Laguerre polynomials of type $(n-1).$

Conversely, there exists a fixed constant $V$ such that if $h$ is
a holomorphic function on the domain
$\ds{\Omega_{q,\frac{2q}{V}}},$ $h^{\la}$ is compactly supported
as a function of $\la$ and for each $r<q$ $$\int_K \int_K
\int_{|Im (z,w)|=r} \int_{\mathbb{H}^n} |h(X \cdot
(z,w,\tau,k_1\exp{(iH)}k_2))|^2 dX d\mu_r dk_1 dk_2 < \infty,$$
then for every $p<q,$ there exists $f \in L^2(\mathbb{HM})$ such
that $h = e^{-p\Delta^{\frac{1}{2}}} f.$
\end{thm}

\begin{proof}
First, we prove the holomorphicity of
$e^{-q\Delta^{\frac{1}{2}}}f$ on
$\Omega_{\frac{p}{2},\frac{p}{N\sqrt{2}}}$ for $0<p<q$ by proving
uniform convergence of the same on compact subsets. So, we
consider a compact subset $M \subseteq
\Omega_{\frac{p}{2},\frac{p}{N\sqrt{2}}}.$

Since \beas e^{-q((2m+n)|\la| + \la^2 +
\lambda_{\pi})^{\frac{1}{2}}} &\leq& e^{-\frac{q
\sqrt{\lambda_{\pi}}}{\sqrt{2}}}
e^{-\frac{q}{\sqrt{2}}((2m+n)|\la| + \la^2)^{\frac{1}{2}}} \\
&\leq& e^{-\frac{q \sqrt{\lambda_{\pi}}}{\sqrt{2}}}
e^{-\frac{q}{2}((2m+n)|\la|)^{\frac{1}{2}}} e^{-\frac{q
|\la|}{2}},\eeas for $(z,w,\tau,g) = (x+iy,u+iv,t+is,ke^{iH}) \in
M \subset \mathbb{C}^{n} \times \mathbb{C}^n \times \mathbb{C}
\times G,$ we have \beas && \left|
e^{-q\Delta^{\frac{1}{2}}}f(z,w,\tau,g) \right| \\
&\leq& C \sum_{\pi \in \widehat{K}} d_{\pi} \sum_{i,j = 1}^{d_\pi}
|\phi_{ij}^{\pi}(g)| e^{-\frac{q \sqrt{\lambda_{\pi}}}{\sqrt{2}}}
\int_{\mathbb{R}} \left( \sum_{m=0}^{\infty} \sum_{a=1}^{A_m}
e^{-\frac{q}{2}((2m+n)|\la|)^{\frac{1}{2}}}
\left|(f_{ij}^{\pi})^{\la} *_{\la} \varphi_{ma}^{\la}(z,w) \right| \right) \\
&& e^{|\la|(|s|-\frac{q}{2})} |\la|^n d\la . \eeas Now, it can be
seen that for a fixed $\la,$ $$ |\varphi_{ma}^{\la}(z,w)|^2 \leq C
\dim \mathcal{P}_{ma} e^{\la(u \cdot y - v \cdot x)}
\varphi_{ma}^{\la}(2iy,2iv).$$ It follows that $$
\left|(f_{ij}^{\pi})^{\la} *_{\la} \varphi_{ma}^{\la}(z,w) \right|
\leq e^{\frac{\la}{2}(u \cdot y - v \cdot x)} \|
(f_{ij}^{\pi})^{\la}\|_1 \left( \dim \mathcal{P}_{ma}
\right)^{\frac{1}{2}} \left( \varphi_{ma}^{\la}(2iy,2iv)
\right)^{\frac{1}{2}}.$$ So, we get that \beas &&
\sum_{m=0}^{\infty} \sum_{a=1}^{A_m}
e^{-\frac{q}{2}((2m+n)|\la|)^{\frac{1}{2}}}
\left|(f_{ij}^{\pi})^{\la} *_{\la} \varphi_{ma}^{\la}(z,w) \right| \\
&\leq& e^{\frac{\la}{2}(u \cdot y - v \cdot x)} \|
(f_{ij}^{\pi})^{\la}\|_1 \sum_{m=0}^{\infty} \sum_{a=1}^{A_m}
e^{-\frac{q}{2}((2m+n)|\la|)^{\frac{1}{2}}} \left( \dim
\mathcal{P}_{ma} \right)^{\frac{1}{2}} \left(
\varphi_{ma}^{\la}(2iy,2iv) \right)^{\frac{1}{2}}. \eeas Applying
Cauchy-Schwarz inequality to the above and noting that
$\ds{\varphi_{m}^{\la}(2iy,2iv) = }$ $\ds{\sum_{a=1}^{A_m}
\varphi_{ma}^{\la}(2iy,2iv)}$ (see \cite{SR} for details) and
$\ds{\dim \mathcal{P}_{m} = }$ $\ds{\sum_{a=1}^{A_m} \dim
\mathcal{P}_{ma} = \frac{(m+n-1)!}{m! (n-1)!}}$ we get that \beas
&& \sum_{m=0}^{\infty} \sum_{a=1}^{A_m}
e^{-\frac{q}{2}((2m+n)|\la|)^{\frac{1}{2}}}
\left|(f_{ij}^{\pi})^{\la} *_{\la} \varphi_{ma}^{\la}(z,w) \right| \\
&\leq& e^{\frac{\la}{2}(u \cdot y - v \cdot x)} \|
(f_{ij}^{\pi})^{\la}\|_1 \sum_{m=0}^{\infty}
e^{-\frac{q}{2}((2m+n)|\la|)^{\frac{1}{2}}} \left(
\frac{(m+n-1)!}{m! (n-1)!} \right)^{\frac{1}{2}} \left(
\varphi_{m}^{\la}(2iy,2iv) \right)^{\frac{1}{2}}. \eeas As in the
proof of Theorem 5.1 of \cite{T1}, for any fixed $(y,v)$ with
$|y|^2 + |v|^2 \leq r^2 < \frac{p^2}{4} < \frac{q^2}{4},$ the
above series is bounded by a constant times $$ \sum_{m=0}^{\infty}
m^{\frac{n-1}{2}} m^{\frac{n-1}{4}-\frac{1}{8}}
e^{-((2m+n)|\la|)^{\frac{1}{2}}(\frac{q}{2}-r)}$$ which certainly
converges if $\ds{r<\frac{p}{2}<\frac{q}{2}.}$ Moreover, using the
fact that $\ds{\|(f_{ij}^{\pi})^{\la}\|_1 \leq \|f\|_1}$ and
$f^{\la}$ is compactly supported as a function of $\la,$ we can
conclude that \beas && \left|
e^{-q\Delta^{\frac{1}{2}}}f(z,w,\tau,g) \right| \\
&\leq& C \sum_{\pi \in \widehat{K}} d_{\pi} \sum_{i,j = 1}^{d_\pi}
|\phi_{ij}^{\pi}(g)| e^{-\frac{q \sqrt{\lambda_{\pi}}}{\sqrt{2}}}.
\eeas For $g = ke^{iH},$ we have $\ds{\phi_{ij}^{\pi}(ke^{iH})=
\sum_{l=1}^{d_{\pi}} \phi_{il}^{\pi}(k) \phi_{lj}^{\pi}(e^{iH}).}$
Since $\pi(k)$ is unitary for $k \in K$ and $\pi(e^{iH})$ is
self-adjoint for $H \in \underline{h},$ it follows that
$\ds{\sum_{l=1}^{d_{\pi}} |\phi_{il}^{\pi}(k)|^2=1}$ and \beas
\sum_{l,j=1}^{d_{\pi}} |\phi_{lj}^{\pi}(e^{iH})|^2 &=&
\sum_{l,j=1}^{d_{\pi}} \left\langle {\pi}(e^{iH}) e_l, e_j
\right\rangle
\overline{\left\langle {\pi}(e^{iH}) e_l, e_j \right\rangle} \\
&=& \sum_{l,j=1}^{d_{\pi}} \left\langle {\pi}(e^{iH}) e_l, e_j
\right\rangle \left\langle e_j,{\pi}(e^{iH}) e_l \right\rangle \\
&=& \sum_{l=1}^{d_{\pi}} \left\langle {\pi}(e^{iH}) e_l,
{\pi}(e^{iH}) e_l \right\rangle
\\ &=&  \chi_{\pi}(\exp 2iH)\eeas where $e_1, e_2, \cdots , e_{d_{\pi}}$
is an orthonormal basis of the Hilbert space $\mathcal{H}_{\pi}$
on which $\pi$ acts. Now, using Cauchy-Schwarz inequality we get
that \beas && \left|
e^{-q\Delta^{\frac{1}{2}}}f(z,w,\tau,g) \right| \\
&\leq& C \sum_{\pi \in \widehat{K}} d_{\pi}^{\frac{5}{2}} \left(
\chi_{\pi}(\exp {2iH}) \right)^{\frac{1}{2}} e^{-\frac{q
\sqrt{\lambda_{\pi}}}{\sqrt{2}}}.\eeas From Lemma 6 and 7 of
\cite{H} we know that there exist constants $A,$ $B,$ $C$ and $M$
such that $\lambda_{\pi} \geq A|\mu|^2,$ $d_{\pi} \leq
B(1+|\mu|^C)$ and $|\chi_{\pi}(\exp{iY})| \leq d_{\pi}
e^{M|Y||\mu|}$ where $\mu$ is the highest weight of $\pi.$ Hence
we have $$ |\chi_{\pi}(\exp 2iH)| \leq d_{\pi} e^{2M|H||\mu|} \leq
d_{\pi} e^{2N|H| \sqrt{\lambda_{\pi}}}$$ where
$N=\frac{M}{\sqrt{A}}.$ It follows that \beas && \left|
e^{-q\Delta^{\frac{1}{2}}}f(z,w,\tau,g) \right| \\
&\leq& C \sum_{\pi \in \widehat{K}} B^3 \left( 1+ \left(
\frac{\la_{\pi}}{A} \right)^{\frac{C}{2}} \right)^3
e^{\sqrt{\lambda_{\pi}}\left( N|H|-\frac{q}{\sqrt{2}} \right)
}\eeas which is finite as long as $\ds{|H|<\frac{q}{N\sqrt{2}}}.$
Hence we have proved that $e^{-q\Delta^{\frac{1}{2}}} f$ extends
to a holomorphic function on the domain
$\Omega_{\frac{p}{2},\frac{p}{N\sqrt{2}}}$ for $p<q.$

Now, we prove the equality in Theorem \ref{poisson}. It should be
noted that the domain $\Omega_{\frac{p}{2},\frac{p}{N\sqrt{2}}}$
is invariant under left translation by the Heisenberg motion group
$\mathbb{HM}.$ For $X=(x',u',t',k') \in \mathbb{HM},$
$(z,w,\tau,g) = (x+iy,u+iv,t+is,k_1\exp{(iH)}k_2) \in
\Omega_{\frac{p}{2},\frac{p}{N\sqrt{2}}} \subseteq \mathbb{C}^{n}
\times \mathbb{C}^n \times \mathbb{C} \times G$ and a function $F$
holomorphic on $\Omega_{\frac{p}{2},\frac{p}{N\sqrt{2}}},$ by
Gutzmer's formula on $K$ we have \beas && \int_K \int_K \int_{|Im
(z,w)|=r} \int_{\mathbb{H}^n}
|F(X \cdot (z,w,\tau,k_1\exp{(iH)}k_2))|^2 dX d\mu_r dk_1 dk_2 \\
&=& \sum_{\pi \in \widehat{K}} d_{\pi} \sum_{i,j = 1}^{d_\pi}
\int_{|Im (z,w)|=r} \int_{\mathbb{R}} \int_{\mathbb{R}^n}
\int_{\mathbb{R}^n}
|F_{ij}^{\pi}(x'+iy,u'+iv,t'+i(s+\frac{1}{2}(u' \cdot y - v \cdot
x')))|^2 \\ && \chi_{\pi}(\exp 2iH) dx' du' dt' d\mu_r.\eeas It is
easy to see that $\ds{\frac{1}{\dim \mathcal{P}_{ma}} \int_{U(n)}
\varphi_{ma}^{\la}(k \cdot(x,u)) dk}$ is a $U(n)$-spherical
function. So it is obvious that $\ds{ \frac{1}{\dim
\mathcal{P}_{ma}} \int_{U(n)} \varphi_{ma}^{\la}(k \cdot(x,u)) dk
= \frac{1}{\dim \mathcal{P}_{m}} \varphi_{m}^{\la}(x,u)}.$ By
analytic continuation on both sides we get $$\ds{\frac{1}{\dim
\mathcal{P}_{ma}} \int_{U(n)} \varphi_{ma}^{\la}(k \cdot(2iy,2iv))
dk = \frac{1}{\dim \mathcal{P}_{m}}}
\ds{\varphi_{m}^{\la}(2iy,2iv)}.$$ Hence it follows that the
integral over $U(n)$ can be seen as an integral over the sphere
$|y|^2+|v|^2=r^2$ such that $$ \frac{1}{\dim \mathcal{P}_{ma}}
\int_{|y|^2+|v|^2=r^2} \varphi_{ma}^{\la}(2iy,2iv) d\mu_r =
\frac{1}{\dim \mathcal{P}_{m}} L_m^{n-1}(-2\la r^2) e^{\la r^2}.$$
So, from Theorem \ref{Gutzmer} we have \beas &&
\int_{|y|^2+|v|^2=r^2} \int_{\mathbb{R}^{2n}}
|(F_{ij}^{\pi})^{\la}(x+iy,u+iv)|^2 e^{\la(u \cdot y - v \cdot x)}
dx du d\mu_r
\\ &=& \sum_{m=0}^{\infty} \sum_{a=1}^{A_m} (\dim
\mathcal{P}_{m})^{-1} L_m^{n-1}(-2\la r^2) e^{\la r^2} \|
(F_{ij}^{\pi})^{\la} *_{\la} \varphi_{ma}^{\la} \|^2 \eeas It
follows that \beas && \int_K \int_K \int_{|Im (z,w)|=r}
\int_{\mathbb{H}^n}
|F(X \cdot (z,w,\tau,k_1\exp{(iH)}k_2))|^2 dX d\mu_r dk_1 dk_2 \\
&=& \sum_{\pi \in \widehat{K}} d_{\pi} \chi_{\pi}(\exp 2iH)
\sum_{m=0}^{\infty} \sum_{a=1}^{A_m} (\dim \mathcal{P}_{m})^{-1}
\int_{\mathbb{R}} L_m^{n-1}(-2\la r^2) e^{\la r^2} e^{2 \la s} \\
&&  \left( \sum_{i,j = 1}^{d_\pi} \| (F_{ij}^{\pi})^{\la} *_{\la}
\varphi_{ma}^{\la} \|^2 \right) ~ d\la. \eeas Hence for $h =
e^{-q\Delta^{\frac{1}{2}}} f$ we get the first part of Theorem
\ref{poisson}.

To prove the converse, we first show that for any $0 < \vartheta <
\infty,$ there exist constants $U,V$ such that \bea \label{e7}
\int_{|H|=\vartheta} \chi_{\pi} (\exp {2iH})
d\sigma_{\vartheta}(H) \geq d_{\pi} U
e^{V\vartheta\sqrt{\lambda_{\pi}}}\eea where
$d\sigma_{\vartheta}(H)$ is the normalized surface measure on the
sphere $\{ H \in \underline{h} : |H| = {\vartheta}\} \subseteq
\mathbb{R}^m$ and $m = \dim \underline{h}.$ If $H \in
\underline{a},$ then there exists a non-singular matrix $Q$ and
pure-imaginary valued linear forms $\nu_1, \nu_2, \cdots ,
\nu_{d_{\pi}}$ on $\underline{a}$ such that $$ Q \pi(H) Q^{-1} =
diag(\nu_1(H), \nu_2(H), \cdots , \nu_{d_{\pi}}(H))$$ where
$diag(a_1, a_2, \cdots , a_k)$ denotes $k \times k$ diagonal
matrix with diagonal entries $a_1, a_2, \cdots , a_k.$ Now,
$\nu(H)=i\langle \nu, H\rangle$ where $\nu$ is a weight of $\pi.$
Then $$ \exp (2i Q \pi(H) Q^{-1}) = Q \exp(2i\pi(H)) Q^{-1} =
diag(e^{2i\nu_1(H)}, e^{2i\nu_2(H)}, \cdots ,
e^{2i\nu_{d_{\pi}}(H)}).$$ Hence \beas \chi_{\pi}(\exp2iH) &=& Tr
(Q \exp(2i\pi(H)) Q^{-1}) \\ &=& e^{-2\langle \nu_1, H\rangle} +
e^{-2\langle \nu_2, H\rangle} + \cdots + e^{-2\langle
\nu_{d_{\pi}}, H\rangle} \\ &\geq& e^{-2\langle \mu,
H\rangle}\eeas where $\mu$ is the highest weight corresponding to
$\pi.$ Integrating the above over $|H|={\vartheta}$ we get \beas
\int_{|H|={\vartheta}} \chi_{\pi} (\exp {2iH})
d\sigma_{\vartheta}(H) &\geq&
\int_{|H|={\vartheta}}  e^{-2\langle \mu, H\rangle} d\sigma_{\vartheta}(H)\\
&=&
\frac{J_{\frac{m}{2}-1}(2i{\vartheta}|\mu|)}{(2i{\vartheta}|\mu|)^{\frac{m}{2}-1}}
\\ &\geq& U e^{{\vartheta}|\mu|}\eeas where $J_{\frac{m}{2}-1}$ is the Bessel function of
order ${(\frac{m}{2}-1)}.$ It is known that $\lambda_{\pi} \approx
|\mu|^2,$ hence we have \bea \label{ch} \int_{|H|={\vartheta}}
\chi_{\pi} (\exp {2iH}) d\sigma_{\vartheta}(H) \geq U ~
e^{V{\vartheta}\sqrt{\lambda_{\pi}}}\eea for some $V.$ Consider
the domain $\ds{\Omega_{q,\frac{2q}{V}}}$ for this $V.$ Let $h$ be
a holomorphic function on the domain
$\ds{\Omega_{q,\frac{2q}{V}}}$ such that $h^{\la}$ is compactly
supported as a function of $\la$ and for $r< q,$
$$\int_K \int_K \int_{|Im (z,w)|=r} \int_{\mathbb{H}^n} |h(X \cdot
(z,w,\tau,k_1\exp{(iH)}k_2))|^2 dX d\mu_r dk_1 dk_2 < \infty.$$
So, as before it follows that \beas && \sum_{\pi \in \widehat{K}}
d_{\pi} \chi_{\pi}(\exp 2iH) \sum_{m=0}^{\infty} \sum_{a=1}^{A_m}
(\dim \mathcal{P}_{m})^{-1}
\int_{\mathbb{R}} L_m^{n-1}(-2\la r^2) e^{\la r^2} e^{2 \la s} \\
&& \left( \sum_{i,j = 1}^{d_\pi} \| (h_{ij}^{\pi})^{\la} *_{\la}
\varphi_{ma}^{\la} \|^2 \right) ~ d\la ~ < ~ \infty ~~~ \txt{ for
} r< q. \eeas Integrating over $|H|={\vartheta}$ for $\vartheta <
\ds{\frac{2q}{V}}$ and using (\ref{ch}), it also follows that
\beas && \sum_{\pi \in \widehat{K}} d_{\pi}
e^{V{\vartheta}\sqrt{\lambda_{\pi}}} \sum_{m=0}^{\infty}
\sum_{a=1}^{A_m} (\dim \mathcal{P}_{m})^{-1}
\int_{\mathbb{R}} L_m^{n-1}(-2\la r^2) e^{\la r^2} e^{2 \la s} \\
&& \left( \sum_{i,j = 1}^{d_\pi} \| (h_{ij}^{\pi})^{\la} *_{\la}
\varphi_{ma}^{\la} \|^2 \right) ~ d\la ~ < ~ \infty ~~~ \txt{ for
} r< q. \eeas Now, Perron's formula (Theorem 8.22.3 of \cite{Sz})
gives
$$ L_m^{\alpha}(\zeta) = \frac{1}{2} \pi^{-\frac{1}{2}}
e^{\frac{\zeta}{2}} (-\zeta)^{-\frac{\alpha}{2} - \frac{1}{4}}
m^{\frac{\alpha}{2} - \frac{1}{4}} e^{2(-m \zeta)^{\frac{1}{2}}}
\left( 1 + O(m^{-\frac{1}{2}}) \right)$$ valid for $\zeta$ in the
complex plane cut along the positive real axis. Note that we
require the formula when $\zeta < 0.$ So, using the fact that
$\ds{\dim \mathcal{P}_{m} = \frac{(m+n-1)!}{m!(n-1)!}}$ and
Perron's formula we get that $$ \sum_{\pi \in \widehat{K}} d_{\pi}
e^{V{\vartheta}\sqrt{\lambda_{\pi}}} \sum_{m=0}^{\infty}
\int_{\mathbb{R}} |\la|^{2n} e^{2 \varsigma
((2m+n)|\la|)^{\frac{1}{2}}} e^{2 \la s} \sum_{i,j = 1}^{d_\pi} \|
(h_{ij}^{\pi})^{\la} *_{\la} \varphi_{ma}^{\la} \|^2 d\la ~ < ~
\infty$$ for $\varsigma < r< q$ and $\vartheta <
\ds{\frac{2q}{V}}.$ For $p<q,$ defining $(f_{ij}^{\pi})^{\la}$ by
$(f_{ij}^{\pi})^{\la} = e^{2p \left( (2m+1)|\la| + \la^2 +
\la_{\pi} \right)^{\frac{1}{2}}}$ $(h_{ij}^{\pi})^{\la}$ and using
the inequality $$\ds{e^{2p \left( (2m+1)|\la| + \la^2 + \la_{\pi}
\right)^{\frac{1}{2}}} \leq e^{2p
\left((2m+1)|\la|\right)^{\frac{1}{2}}} e^{2p|\la|} e^{2p
\sqrt{\la_{\pi}}}}$$ we obtain
$$ f(x,u,t,k) = \sum_{\pi \in \widehat{K}} d_{\pi} \sum_{i,j =
1}^{d_\pi} f_{ij}^{\pi}(x,u,t) \phi_{ij}^{\pi}(k) \in
L^2(\mathbb{HM})$$ and $h = e^{-p\Delta^{\frac{1}{2}}} f.$
\end{proof}

\subsection{A generalization of the Segal-Bargmann
transform}\hspace{0.3in}

In \cite{H} Brian C. Hall proved the following generalizations of
the Segal-Bargmann transfoms for $\mathbb{R}^n$ and compact Lie
groups :

\begin{thm}\label{hall}

\item [(I)] Let $\mu$ be any measurable function on $\mathbb{R}^n$
such that
\begin{itemize}
\item $\mu$ is strictly positive and locally bounded away from
zero, \item $ \ds{ \forall ~  x \in \mathbb{R}^n, ~~ \sigma(x) =
\int_{\mathbb{R}^n} e^{2x \cdot y} \mu(y) dy < \infty .}$
\end{itemize}
Define, for $z \in \mathbb{C}^n$ $$ \psi(z) = \int_{\mathbb{R}^n}
\frac{e^{ia(y)}}{\sqrt{\sigma(y)}} e^{-iy \cdot z} dy,$$ where $a$
is a real valued measurable function on $\mathbb{R}^n.$ Then the
mapping $C_{\psi} : L^2(\mathbb{R}^n) \rightarrow
\mathcal{O}(\mathbb{C}^n)$ defined by $$ C_{\psi}(z) =
\int_{\mathbb{R}^n} f(x) \psi(z-x) dx $$ is an isometric
isomorphism of $L^2(\mathbb{R}^n)$ onto $\mathcal{O}(\mathbb{C}^n)
\bigcap L^2(\mathbb{C}^n,dx \mu(y)dy).$

\item [(II)] Let $K$ be a compact Lie group and $G$ be its
complexification. Let $\nu$ be a measure on $G$ such that
\begin{itemize}
\item $\nu$ is bi$-K$-invariant, \item $\nu$ is given by a
positive density which is locally bounded away from zero, \item
For each irreducible representation $\pi$ of $K,$ analytically
continued to $G,$ $$ \delta(\pi) = \frac{1}{dim V_{\pi}} \int_{G}
\|\pi(g^{-1})\|^2 d\nu(g) < \infty.$$
\end{itemize}
Define $ \ds {\tau (g) = \sum _ {\pi \in \widehat{K}}
\frac{d_{\pi}}{\sqrt{\delta(\pi)}} Tr (\pi(g^{-1}) U_{\pi})} $
where $g \in G$ and $U_{\pi}$'s are arbitrary unitary matrices.
Then the mapping $$C_\tau f(g) = \int_{K} f(k) \tau(k^{-1}g)dk$$
is an isometric isomorphism of $L^2(K)$ onto $\ds{\mathcal{O}(G)
\bigcap L^2(G, d\nu(w)).}$

\end{thm}

In a similar fashion, we prove a generalization of Theorem
\ref{sbt} for $\mathbb{HM}.$ For each non-zero $\la \in
\mathbb{R},$ let $W_{\la}$ be a $K$-invariant measurable function
on $\mathbb{R}^{2n}$ such that it satisfies the following
conditions : \begin{itemize} \item $W_{\la}$ is strictly positive
and locally bounded away from zero uniformly in $\la,$ \item For
each $0 \leq m<\infty$ and $1 \leq a \leq A_m,$ $$\ds{
\sigma_{m,a,\la} = \int_{\mathbb{R}^{2n}}
\varphi_{ma}^{\la}(2iy,2iv) W_{\la}(y,v) dy dv < \infty .}$$
\end{itemize} For $x, u \in \mathbb{R}^n,$ define $$ p^{\la}(x,u) =
\left(\frac{2\pi}{|\la|}\right)^{-\frac{n}{2}} \sum_{m=0}^{\infty}
\sum_{a=1}^{A_m} C_{m,a,\la} \left( \sigma_{m,a,\la}
\right)^{-\frac{1}{2}} \left( \dim \mathcal{P}_{ma}
\right)^{\frac{1}{2}} \varphi_{ma}^{\la}(x,u),$$ where
$|C_{m,a,\la}|=1.$ For $(x,u,t) \in \mathbb{H}^n,$ consider
$$p(x,u,t) = (4\pi)^{-n} \int_{\mathbb{R}} e^{-i\la t} p^{\la}(x,u)
e^{-\la^2} d\la.$$ Next, let $\nu, ~ \delta $ and $\tau$ be as in
Theorem \ref{hall} (II). Also define $\ds{\psi(X,k) =
p(X)\tau(k)}$ for $X \in \mathbb{H}^n, ~ k \in K.$ Since
$\sigma_{m,a,\la}$ has exponential growth, it can be proved in a
manner similar to Theorem \ref{sbt} and Theorem \ref{poisson} that
$\psi$ extends to a holomorphic function on $\mathbb{H}^n \times
G.$ Let $\mathcal{A}^{\la}(\mathbb{C}^{2n} \times G)$ be the
weighted Bergman space $\ds{\mathcal{A}^{\la}(\mathbb{C}^{2n}
\times G) = \{ F}$ entire
$$\txt{on } \ds{\mathbb{C}^{2n} \times G : \int_G
\int_{\mathbb{C}^{2n}} |F(x+iy,u+iv,g)|^2 W_{\lambda}(y,v)
e^{\la(u \cdot y - v \cdot x)} ~ dx ~ du ~ dy ~ dv ~ d\nu(g) <
\infty \}.}$$ Then using Theorem \ref{Gutzmer} it is easy to prove
the following analogue of Theorem \ref{hall} for $\mathbb{HM}$
following the methods in Theorem \ref{sbt}.

\begin{thm}
The mapping $$\ds{C_\psi f(Z,g) = \int_{\mathbb{HM}} f(X, k)
\phi((X,k)^{-1}(Z,g))dX dk}$$ is an isometric isomorphism of
$L^2(\mathbb{HM})$ onto $\ds{\int_{\mathbb{R}^*}^{\oplus}
\mathcal{A}^{\la}(\mathbb{C}^{2n} \times G) e^{2\la^2} d\la}.$
\end{thm}

\vspace{0.8 in}

\section{Plancherel Theorem and complexified representations}

In this section, we first state and prove the Plancherel theorem
for Heisenberg motion groups. Thereby, we also list all the
irreducible unitary representations of $\mathbb{HM}$ which occur
in the Plancherel theorem. We then use these representations to
prove a Paley-Wiener type theorem, which is inspired by Theorem
3.1 of \cite{G2}.

\subsection{Representations of $\mathbb{HM}$ and Plancherel
theorem}\hspace{0.2in}

Let $(\sigma,\mathcal{H}_{\sigma})$ be any irreducible, unitary
representation of $K.$ For each $\la \neq 0$ and $\sigma \in
\widehat{K},$ we consider the representations
$\ds{\rho_{\sigma}^{\la}}$ of $\mathbb{HM}$ on the tensor product
space $\ds{ L^2(\mathbb{R}^n) \otimes \mathcal{H}_{\sigma}}$
defined by $$ \rho_{\sigma}^{\la}(x,u,t,k) = \pi_{\la}(x,u,t)
\mu_{\la}(k) \otimes \sigma(k)$$ where $\pi_{\la}$ and $\mu_{\la}$
are the Schr$\ddot{\txt{o}}$dinger and metaplectic representations
respectively and $(x,u,t,k) \in \mathbb{HM}.$

\begin{prop}
Each $\ds{\rho_{\sigma}^{\la}}$ is unitary and irreducible.
\end{prop}

\begin{proof}
It is easily seen that each $\ds{\rho_{\sigma}^{\la}}$ is unitary.
We shall now prove that $\ds{\rho_{\sigma}^{\la}}$ is irreducible.
Suppose $M \subset \ds{ L^2(\mathbb{R}^n) \otimes }$
$\ds{\mathcal{H}_{\sigma}}$ is invariant under all
$\ds{\rho_{\sigma}^{\la}}(x,u,t,k).$ If $M \neq \{ 0 \}$ we will
show that $M= \ds{ L^2(\mathbb{R}^n) \otimes
\mathcal{H}_{\sigma}}$ proving the irreducibility of
$\ds{\rho_{\sigma}^{\la}}.$ If $M$ is a proper subspace of $\ds{
L^2(\mathbb{R}^n) \otimes \mathcal{H}_{\sigma}}$ invariant under
$\ds{\rho_{\sigma}^{\la}}(x,u,t,k)$ for all $(x,u,t,k),$ then
there are nontrivial elements $f$ and $g$ in $\ds{
L^2(\mathbb{R}^n) \otimes \mathcal{H}_{\sigma}}$ such that $f \in
M$ and $g$ is orthogonal to $\ds{\rho_{\sigma}^{\la}(x,u,t,k)f}$
for all $(x,u,t,k).$ This means that $\ds{\langle
\rho_{\sigma}^{\la}(x,u,t,k)f, g \rangle = 0}$ for all
$(x,u,t,k).$

Recall the functions $\phi_{\alpha}^{\la}$ from Section
\ref{hermite}. It is easily seen that for each $\la \neq 0,$
$\{\phi_{\alpha}^{\la} : \alpha \in \mathbb{N}^n \}$ forms an
orthonormal basis for $L^2(\mathbb{R}^n).$ Then, an orthonormal
basis of $\ds{ L^2(\mathbb{R}^n) \otimes \mathcal{H}_{\sigma}}$ is
given by $\ds{\{ \phi_{\alpha}^{\la} \otimes e_{i}^{\sigma} :
\alpha \in \mathbb{N}^n, 1 \leq i \leq d_{\sigma}\}}$ where
$\ds{\{ e_{i}^{\sigma} : 1 \leq i \leq d_{\sigma}\}}$ is an
orthonormal basis of $\mathcal{H}_{\sigma}$ and $d_{\sigma} = \dim
\mathcal{H}_{\sigma}.$ Now, given $f, ~ g \in \ds{
L^2(\mathbb{R}^n) \otimes \mathcal{H}_{\sigma}},$ consider the
function $$\ds{V_g^f(x,u,t,k) = \langle
\rho_{\sigma}^{\la}(x,u,t,k)f, g \rangle}.$$ From the discussion
in Section \ref{s1} it follows that \bea \label{met} \mu_{\la}(k)
\phi_{\gamma}^{\la} = \sum_{|\alpha| = |\gamma|} \eta_{\alpha
\gamma}^{\la}(k) \phi_{\alpha}^{\la} \eea where $\eta_{\alpha
\gamma}^{\la}$'s are the matrix coefficients of $\mu_{\la}$ and $k
\in K \subseteq U(n).$ It is to be noted that this expansion is in
terms of the scaled Hermite functions $\phi_{\alpha}^{\la}$ and
not in terms of the modified $K$-Hermite functions
$\psi_{\alpha}^{\la}$ defined in Section \ref{kfns}. So the
summation is taken over the whole of $|\alpha|=|\gamma|$ and not
over a particular $\mathcal{P}_{ma}.$ Hence, it follows that
$$ V_g^f(x,u,t,k) = (2\pi)^{\frac{n}{2}} |\la|^{-\frac{n}{2}} e^{i
\la t} \sum_{\alpha, \beta \in \mathbb{N}^n} \sum_{1 \leq i,j \leq
d_{\sigma}} \sum_{|\gamma| = |\alpha|} f_{\gamma,i} ~
\overline{g_{\beta,j}} \eta_{\alpha \gamma}^{\la}(k) \phi_{\alpha
\beta}^{\la}(x,u) \phi_{ji}^{\sigma}(k)$$ where $\ds{f =
\sum_{\gamma \in \mathbb{N}^n} \sum_{1 \leq i \leq d_{\sigma}}
f_{\gamma,i} ~ \phi_{\gamma}^{\la} \otimes e_{i}^{\sigma}},$
$\ds{g = \sum_{\beta \in \mathbb{N}^n} \sum_{1 \leq j \leq
d_{\sigma}} g_{\beta,j} ~ \phi_{\beta}^{\la} \otimes
e_{j}^{\sigma}}$ and $\phi_{ji}^{\sigma}$ are the matrix
coefficients of $\sigma.$ Then calculating the $L^2$ norm of
$V_g^f$ with respect to $x,u$ we get
$$\int_{\mathbb{R}^{2n}} |V_g^f(x,u,t,k)|^2 dx du = C
\sum_{\alpha, \beta \in \mathbb{N}^n} \left|\sum_{1 \leq i,j \leq
d_{\sigma}} \sum_{|\gamma| = |\alpha|} f_{\gamma,i} ~
\overline{g_{\beta,j}} \eta_{\alpha \gamma}^{\la}(k)
\phi_{ji}^{\sigma}(k)\right|^2. $$

Now, for $\pi \in \widehat{K},$ if $\mathcal{H}_{\pi}$ is the
representation space of $\pi$ and $v_1, v_2, \cdots , v_{d_{\pi}}$
is a basis of $\mathcal{H}_{\pi},$ then for complex numbers $c_i,
1 \leq i \leq d_{\pi}$ and $u \in K,$ we have
\begin{eqnarray}\label{pr1}\begin{array}{rcll} &&
\sum_{q=1}^{d_{\pi}}
\left|\sum_{i=1}^{d_{\pi}} c_i \phi_{qi}^{\pi}(u) \right|^2 \\
&=& \sum_{q=1}^{d_{\pi}} \sum_{i=1}^{d_{\pi}} c_i
\phi_{qi}^{\pi}(u) \sum_{a=1}^{d_{\pi}} \overline{c_a}
\overline{\phi_{qa}^{\pi}(u)} \\
&=& \sum_{i,a=1}^{d_{\pi}} c_i \overline{c_a} \sum_{q=1}^{d_{\pi}}
\left\langle \pi(u)v_i ,v_q \right\rangle \left\langle
v_q , \pi(u)v_a \right\rangle \\
&=& \sum_{i,a=1}^{d_{\pi}} c_i \overline{c_a} \left\langle
\pi(u)v_i , \pi(u)v_a \right\rangle \\
&=& \sum_{i=1}^{d_{\pi}} |c_i|^2, \end{array} \end{eqnarray} since
$\pi$ is a unitary representation of $K.$ Hence we obtain
$$\int_{\mathbb{R}^{2n}} |V_g^f(x,u,t,k)|^2 dx du = C
\sum_{\gamma, \beta \in \mathbb{N}^n} \left|\sum_{1 \leq i,j \leq
d_{\sigma}} f_{\gamma,i} ~ \overline{g_{\beta,j}}
\phi_{ji}^{\sigma}(k)\right|^2.$$ Integrating over $K$ we get that
\beas \int_K \int_{\mathbb{R}^{2n}} |V_g^f(x,u,t,k)|^2 dx du dk
&=& C \left(\sum_{\gamma \in \mathbb{N}^n} \sum_{1 \leq i \leq
d_{\sigma}} |f_{\gamma,i}|^2 \right) \left(\sum_{\beta \in
\mathbb{N}^n} \sum_{1 \leq j \leq d_{\sigma}} |g_{\beta,j}|^2
\right)\\ &=& C \|f\|^2 \|g\|^2.\eeas Under our assumption that
$M$ is nontrivial and proper, we have $V_g^f=0$ which means that
$\|f\|^2 \|g\|^2 = 0.$ This is a contradiction since both $f$ and
$g$ are nontrivial. Hence $M$ has to be the whole of $\ds{
L^2(\mathbb{R}^n) \otimes \mathcal{H}_{\sigma}}$ and this proves
that $\ds{\rho_{\sigma}^{\la}}$ is irreducible.
\end{proof}

We now show that the representations $\ds{\rho_{\sigma}^{\la}}$
are enough for the Plancherel theorem. Given $f \in L^1 \bigcap
L^2 (\mathbb{HM})$ consider the group Fourier transform \beas
\widehat{f}(\la,\sigma) &=& \int_{K} \int_{\mathbb{R}}
\int_{\mathbb{R}^{2n}} f(x,u,t,k) \rho_{\sigma}^{\la}(x,u,t,k) dx
du dt dk \\ &=& \int_{K} \int_{\mathbb{R}^{2n}} f^{\la}(x,u,k)
\left(\pi_{\la}(x,u) \mu_{\la}(k) \otimes \sigma(k)\right) dx du
dk. \eeas

\begin{thm}(Plancherel)
For $f \in L^1 \bigcap L^2 (\mathbb{HM})$ we have $$\int_{K}
\int_{\mathbb{H}^n} |f(x,u,t,k)|^2 dx du dt dk = (2\pi)^{-n}
\sum_{\sigma \in \widehat{K}} d_{\sigma} \int_{\mathbb{R}
\diagdown \{0\}} \left\| \widehat{f}(\la,\sigma) \right\|_{HS}^2
|\la|^n d\la.$$
\end{thm}

\begin{proof}
We calculate the Hilbert-Schmidt operator norm of
$\widehat{f}(\la,\sigma)$ by using the basis $\ds{\{
\phi_{\gamma}^{\la} \otimes e_{i}^{\sigma} : \gamma \in
\mathbb{N}^n, 1 \leq i \leq d_{\sigma}\}}.$ We have, by
(\ref{met}), \beas \widehat{f}(\la,\sigma)(\phi_{\gamma}^{\la}
\otimes e_{i}^{\sigma}) = \sum_{|\alpha| = |\gamma|} \int_K
\eta_{\alpha \gamma}^{\la}(k) \int_{\mathbb{R}^{2n}}
f^{\la}(x,u,k) \left(\pi_{\la}(x,u) \phi_{\alpha}^{\la} \otimes
\sigma(k)e_{i}^{\sigma}\right) dx du dk. \eeas Thus \beas &&
\left\langle \widehat{f}(\la,\sigma)(\phi_{\gamma}^{\la} \otimes
e_{i}^{\sigma}) , \phi_{\beta}^{\la} \otimes e_{j}^{\sigma}
\right\rangle \\ &=& (2\pi)^{\frac{n}{2}} |\la|^{-\frac{n}{2}}
\sum_{|\alpha| = |\gamma|} \int_K \eta_{\alpha \gamma}^{\la}(k)
\int_{\mathbb{R}^{2n}} f^{\la}(x,u,k) \phi_{\alpha
\beta}^{\la}(x,u) \phi_{ji}^{\sigma}(k) dx du dk \eeas so that
\beas && (2\pi)^{-n} |\la|^{n} \left\|
\widehat{f}(\la,\sigma)(\phi_{\gamma}^{\la}
\otimes e_{i}^{\sigma}) \right\|^2 \\
&=& \sum_{\beta \in \mathbb{N}^n} \sum_{1 \leq j \leq d_{\sigma}}
\left| \sum_{|\alpha| = |\gamma|} \int_K \eta_{\alpha
\gamma}^{\la}(k) \int_{\mathbb{R}^{2n}} f^{\la}(x,u,k)
\phi_{\alpha \beta}^{\la}(x,u) \phi_{ji}^{\sigma}(k) dx du dk
\right|^2 \eeas and \beas &&
(2\pi)^{-n} |\la|^{n} \left\| \widehat{f}(\la,\sigma) \right\|_{HS}^2 \\
&=& \sum_{\beta, \gamma \in \mathbb{N}^n} \sum_{1 \leq i, j \leq
d_{\sigma}} \left| \sum_{|\alpha| = |\gamma|} \int_K \eta_{\alpha
\gamma}^{\la}(k) \int_{\mathbb{R}^{2n}} f^{\la}(x,u,k)
\phi_{\alpha \beta}^{\la}(x,u) \phi_{ji}^{\sigma}(k) dx du dk
\right|^2.\eeas Using Plancherel theorem for $K,$ we get that
\beas && (2\pi)^{-n} |\la|^{n} \sum_{\sigma \in \widehat{K}}
d_{\sigma} \left\| \widehat{f}(\la,\sigma) \right\|_{HS}^2 \\
&=& \sum_{\beta, \gamma \in \mathbb{N}^n} \int_K \left|
\sum_{|\alpha| = |\gamma|} \eta_{\alpha \gamma}^{\la}(k)
\int_{\mathbb{R}^{2n}} f^{\la}(x,u,k) \phi_{\alpha
\beta}^{\la}(x,u) dx du \right|^2 dk. \eeas Applying the same
arguments as in (\ref{pr1}) we obtain that the above equals \beas
\sum_{\alpha, \beta \in \mathbb{N}^n} \int_K \left|
\int_{\mathbb{R}^{2n}} f^{\la}(x,u,k) \phi_{\alpha
\beta}^{\la}(x,u) dx du \right|^2 dk. \eeas Noting that
$\{\phi_{\alpha \beta}^{\la} : \alpha, \beta \in \mathbb{N}^n\}$
is an orthonormal basis for $L^2(\mathbb{R}^{2n})$ we have \beas
(2\pi)^{-n} |\la|^{n} \sum_{\sigma \in \widehat{K}} d_{\sigma}
\left\| \widehat{f}(\la,\sigma) \right\|_{HS}^2 = \int_K
\int_{\mathbb{R}^{2n}} \left|  f^{\la}(x,u,k) \right|^2 dx du dk.
\eeas Therefore, \beas (2\pi)^{-n} \int_{\mathbb{R}} \left(
\sum_{\sigma \in \widehat{K}} d_{\sigma} \left\|
\widehat{f}(\la,\sigma) \right\|_{HS}^2 \right) |\la|^{n} d\la &=&
\int_{\mathbb{R}} \int_K \int_{\mathbb{R}^{2n}} \left|
f^{\la}(x,u,k) \right|^2 dx du dk d\la \\
&=& \int_{\mathbb{R}} \int_K \int_{\mathbb{R}^{2n}} \left|
f(x,u,t,k) \right|^2 dx du dk dt.\eeas
\end{proof}

Since $L^1 \bigcap L^2(\mathbb{HM})$ is dense in
$L^2(\mathbb{HM}),$ the Fourier transform can be uniquely extended
to the whole of $L^2(\mathbb{HM})$ and the above Plancherel
theorem holds true for the same.

\subsection{Complexified representations and Paley-Wiener type
theorems}\hspace{0.2in}

We know that the operators $\widehat{f}(\la,\sigma)$ act on the
basis elements $\ds{ \phi_{\gamma}^{\la} \otimes e_{i}^{\sigma}},$
which gives \beas && \widehat{f}(\la,\sigma)(\phi_{\gamma}^{\la}
\otimes e_{i}^{\sigma}) \\ &=& \sum_{|\alpha| = |\gamma|} \int_K
\eta_{\alpha \gamma}^{\la}(k') \int_{\mathbb{R}^{2n}}
f^{\la}(x',u',k') \left(\pi_{\la}(x',u') \phi_{\alpha}^{\la}
\otimes \sigma(k')e_{i}^{\sigma}\right) dx' du' dk'. \eeas Now, if
we consider the operator $\ds{\rho_{\sigma}^{\la}(x,u,t,k)
\widehat{f}(\la,\sigma)}$ acting on the basis elements we get that
\beas && \rho_{\sigma}^{\la}(x,u,t,k) \widehat{f}(\la,\sigma)
(\phi_{\gamma}^{\la} \otimes e_{i}^{\sigma}) \\
&=& e^{i \la t} \sum_{|\alpha| = |\gamma|} \int_K \eta_{\alpha
\gamma}^{\la}(k') \int_{\mathbb{R}^{2n}} f^{\la}(x',u',k')
\left(\pi_{\la}(x,u) \mu_{\la}(k) \pi_{\la}(x',u')
\phi_{\alpha}^{\la} \otimes \right. \\ && \left. \sigma(k)
\sigma(k')e_{i}^{\sigma}\right) dx' du' dk'. \eeas Then it follows
from (\ref{meta}) that the above equals \beas && e^{i \la t}
\sum_{|\alpha| = |\gamma|} \int_K \eta_{\alpha \gamma}^{\la}(k')
\int_{\mathbb{R}^{2n}} f^{\la}(x',u',k') \left(\pi_{\la}(x,u)
\pi_{\la}(k \cdot (x',u')) \mu_{\la}(k) \phi_{\alpha}^{\la}
\otimes \right. \\ && \left. \sigma(k)
\sigma(k')e_{i}^{\sigma}\right) dx' du' dk'. \eeas So we obtain
\beas && \rho_{\sigma}^{\la}(x,u,t,k) \widehat{f}(\la,\sigma)
(\phi_{\gamma}^{\la} \otimes e_{i}^{\sigma}) \\
&=& e^{i \la t} \sum_{|\alpha'| = |\alpha|} \sum_{|\alpha| =
|\gamma|} \int_K \eta_{\alpha' \alpha}^{\la}(k) \eta_{\alpha
\gamma}^{\la}(k') \int_{\mathbb{R}^{2n}} f^{\la}(x',u',k')
(\pi_{\la}(x,u) \pi_{\la}(k \cdot(x',u')) \phi_{\alpha'}^{\la} \\
&& \otimes
\sigma(k) \sigma(k')e_{i}^{\sigma}) dx' du' dk' \\
&=& e^{i \la t} \sum_{|\alpha'| = |\gamma|} \int_K \eta_{\alpha'
\gamma}^{\la}(kk') \int_{\mathbb{R}^{2n}} f^{\la}(x',u',k')
(\pi_{\la}(x,u) \pi_{\la}(k \cdot(x',u')) \phi_{\alpha'}^{\la} \\
&& \otimes \sigma(kk')e_{i}^{\sigma}) dx' du' dk'. \eeas Noting
that the action of $K \subseteq U(n)$ on $\mathbb{R}^{2n}$
naturally extends to an action of $G$ on $\mathbb{C}^{2n},$ this
action of $\ds{\rho_{\sigma}^{\la}(x,u,t,k)
\widehat{f}(\la,\sigma)}$ on the basis elements
$\ds{\phi_{\gamma}^{\la} \otimes e_{i}^{\sigma}}$ can clearly be
analytically continued to $\mathbb{HM}_{\mathbb{C}} =
\mathbb{C}^{n} \times \mathbb{C}^n \times \mathbb{C} \times G$ and
we get that (for suitable functions $f$) \beas &&
\rho_{\sigma}^{\la}(z,w,\tau,g) \widehat{f}(\la,\sigma)
(\phi_{\gamma}^{\la} \otimes e_{i}^{\sigma}) \\ &=& e^{i \la
(t+is)} \sum_{|\alpha| = |\gamma|} \int_K \eta_{\alpha
\gamma}^{\la}(ke^{iH}k')
\int_{\mathbb{R}^{2n}} f^{\la}(x',u',k') (\pi_{\la}(x+iy,u+iv) \\
&& \pi_{\la}(ke^{iH} \cdot(x',u')) \phi_{\alpha'}^{\la} \otimes
\sigma(ke^{iH}k')e_{i}^{\sigma}) dx' du' dk' \eeas where
$(z,w,\tau,g) = (x+iy,u+iv,t+is,ke^{iH})\in
\mathbb{HM}_{\mathbb{C}}.$ We have the following theorem:

\begin{thm}\label{thm}
Let $f \in L^2(\mathbb{HM}).$ Then $f$ extends holomorphically to
$\mathbb{HM}_{\mathbb{C}}$ with
$$\int_{\mathbb{H}^n}
|f((z,w,\tau,g)^{-1} X|^2 dX < \infty ~~~~~~ \forall ~~
(z,w,\tau,g) \in \mathbb{HM}_{\mathbb{C}}$$ iff
$$ \sum_{\sigma \in \widehat{K}} d_{\sigma} \int_{\mathbb{R}}
\left\| \rho_{\sigma}^{\la}(z,w,\tau,g) \widehat{f}(\la,\sigma)
\right\|_{HS}^2 |\la|^{n} d\la < \infty.
$$ In this case we also have \bea \label{eql} \begin{array}{rcll}
&& \ds{\int_{\mathbb{H}^n} |f((z,w,\tau,g)^{-1} X|^2 dX }\\ &=&
\ds{(2 \pi)^{-2n} \sum_{\sigma \in \widehat{K}} d_{\sigma}
\int_{\mathbb{R}} \left\| \rho_{\sigma}^{\la}(z,w,\tau,g)
\widehat{f}(\la,\sigma) \right\|_{HS}^2 |\la|^{n} d\la.}
\end{array} \eea
\end{thm}

Before we start the proof let us set up some more notation. We
know that any $f \in L^2(\mathbb{HM})$ can be expanded in the $K$
variable using the Peter Weyl theorem to obtain \bea \label{eqn3}
f(x,u,t,k) = \sum_{\pi \in \widehat{K}} d_{\pi} \sum_{i,j =
1}^{d_\pi} f_{ij}^{\pi}(x,u,t) \phi_{ij}^{\pi}(k)\eea where for
each $\pi \in \widehat{K},$ $d_\pi$ is the degree of $\pi,$
$\phi_{ij}^{\pi}$'s are the matrix coefficients of $\pi$ and
$\ds{f_{ij}^{\pi}(x,u,t) = \int_K f(x,u,t,k)
\overline{\phi_{ij}^{\pi}(k)} dk}.$

Now, for $F \in L^2(\mathbb{R}^{2n}),$ consider the decomposition
of the function $k \mapsto F(k \cdot (x,u))$ from $K$ to
$\mathbb{C}$ in terms of the irreducible unitary representations
of $K$ given by $$F(k \cdot (x,u)) = \sum_{\nu \in \widehat{K}}
d_{\nu} \sum_{p,q=1}^{d_{\nu}} F^{pq}_{\nu}(x,u)
\phi_{pq}^{\nu}(k)$$ where $\ds{F^{pq}_{\nu}(x,u) = \int_K F(k
\cdot (x,u)) \overline{\phi_{pq}^{\nu} (k)} dk}.$ Putting $k = e,$
the identity element of $K,$ we obtain
$$F(x,u) = \sum_{\nu \in \widehat{K}} d_{\nu}
\sum_{p=1}^{d_{\nu}} F^{pp}_{\nu}(x,u).$$ Then it is easy to see
that for $k \in K,$ \bea \label{eqn1} F^{pp}_{\nu}(k \cdot (x,u))
= \sum_{q=1}^{d_{\nu}} F^{pq}_{\nu}(x,u) \phi_{pq}^{\nu}(k).\eea
From the above and the fact that $f_{ij}^{\pi} \in
L^2(\mathbb{H}^n)$ for every $\pi \in \widehat{K}$ and $1 \leq i,j
\leq d_{\pi}$ it follows that any $f \in L^2(\mathbb{HM})$ can be
written as $$ f(x,u,t,k) = \sum_{\pi \in \widehat{K}} d_{\pi}
\sum_{\nu \in \widehat{K}} d_{\nu} \sum_{i,j=1}^{d_{\pi}}
\sum_{p=1}^{d_{\nu}} (f_{ij}^{\pi})^{pp}_{\nu}(x,u,t)
\phi_{ij}^{\pi}(k).$$ In the following lemma we will prove the
theorem for the functions of the form $\ds{\sum_{i,j=1}^{d_{\pi}}
\sum_{p=1}^{d_{\nu}} f^{pp}_{ij}(x,u,t) \phi_{ij}^{\pi}(k)}.$ Then
we shall prove the orthogonality of each part with respect to the
given inner product so that we can sum up in order to prove
Theorem \ref{thm}.

\begin{lem}\label{lem}
For fixed $\pi, \nu \in \widehat{K},$ Theorem \ref{thm} is true
for functions of the form $$f(x,u,t,k) = \sum_{i,j=1}^{d_{\pi}}
\sum_{p=1}^{d_{\nu}} f^{pp}_{ij}(x,u,t) \phi_{ij}^{\pi}(k)$$ where
for simplicity we write $(f_{ij}^{\pi})^{pp}_{\nu}$ as
$f^{pp}_{ij}.$
\end{lem}

\begin{proof}

First we assume that $f \in L^2(\mathbb{HM})$ is holomorphic on
$\mathbb{C}^n \times \mathbb{C}^n \times \mathbb{C} \times G$ with
$$\int_{\mathbb{H}^n} |f((z,w,\tau,g)^{-1} X|^2 dX < \infty ~~~~~~ \forall ~~
(z,w,\tau,g) \in \mathbb{C}^n \times \mathbb{C}^n \times
\mathbb{C} \times G$$ and is of the form $$f(x,u,t,k) =
\sum_{i,j=1}^{d_{\pi}} \sum_{p=1}^{d_{\nu}} f^{pp}_{ij}(x,u,t)
\phi_{ij}^{\pi}(k).$$ For $(x,u,t,k)\in \mathbb{HM},$ we have
\beas && \rho_{\sigma}^{\la}(x,u,t,k) \widehat{f}(\la,\sigma)
(\phi_{\gamma}^{\la} \otimes e_{l}^{\sigma}) \\
&=& e^{i \la t} \sum_{|\alpha| = |\gamma|} \int_K \eta_{\alpha
\gamma}^{\la}(kk') \int_{\mathbb{R}^{2n}} f^{\la}(x',u',k')
(\pi_{\la}(x,u) \pi_{\la}(k \cdot(x',u')) \phi_{\alpha}^{\la} \\
&& \otimes \sigma(kk')e_{l}^{\sigma}) dx' du' dk'. \eeas Making
changes of variables $\ds{(x',u') \ra k^{-1} \cdot (x',u')},$ $k'
\ra k^{-1}k'$ and using the special form of $f$ along with
(\ref{eqn1}) we obtain that the above equals \beas && e^{i \la t}
\sum_{|\alpha| = |\gamma|} \int_K \eta_{\alpha \gamma}^{\la}(kk')
\int_{\mathbb{R}^{2n}} f^{\la}(k^{-1}(x',u'),k') (\pi_{\la}(x,u)
\pi_{\la}(x',u') \phi_{\alpha}^{\la} \\ && \otimes
\sigma(kk')e_{l}^{\sigma}) dx' du' dk' \\ &=& e^{i \la t}
\sum_{|\alpha| = |\gamma|} \sum_{i,j=1}^{d_{\pi}}
\sum_{p,q=1}^{d_{\nu}} \int_K \eta_{\alpha \gamma}^{\la}(k')
\int_{\mathbb{R}^{2n}} (f^{\la})^{pq}_{ij}(x',u')
(\pi_{\la}(x,u) \pi_{\la}(x',u') \phi_{\alpha}^{\la} \\
&& \otimes \sigma(k')e_{l}^{\sigma}) dx' du' ~
\phi_{ij}^{\pi}(k^{-1}k') \phi_{pq}^{\nu}(k^{-1}) dk'.\eeas Then,
for $(z,w,\tau,g) = (x+iy,u+iv,t+is,ke^{iH}) \in \mathbb{C}^n
\times \mathbb{C}^n \times \mathbb{C} \times G,$ we get that \beas
&& \rho_{\sigma}^{\la}(z,w,\tau,g) \widehat{f}(\la,\sigma)
(\phi_{\gamma}^{\la} \otimes e_{l}^{\sigma}) \\ &=& e^{i \la
(t+is)} \sum_{|\alpha| = |\gamma|} \sum_{i,j=1}^{d_{\pi}}
\sum_{p,q=1}^{d_{\nu}} \int_K \int_{\mathbb{R}^{2n}}
(f^{\la})^{pq}_{ij}(x',u')
(\pi_{\la}(x+iy,u+iv) \pi_{\la}(x',u') \phi_{\alpha}^{\la} \\
&& \otimes \sigma(k')e_{l}^{\sigma}) dx' du' ~ \eta_{\alpha
\gamma}^{\la}(k') \phi_{ij}^{\pi}((ke^{iH})^{-1}k')
\phi_{pq}^{\nu}((ke^{iH})^{-1}) dk'. \eeas Thus, expanding the
inner product $\langle \pi_{\la}(z,w) \pi_{\la}(x',u')
\phi_{\alpha}^{\la} , \phi_{\beta}^{\la} \rangle$ in terms of
$\phi_{\delta}^{\la}$ and using (\ref{schr}), we have \beas &&
\left\langle \rho_{\sigma}^{\la}(z,w,\tau,g)
\widehat{f}(\la,\sigma)(\phi_{\gamma}^{\la} \otimes
e_{l}^{\sigma}) , \phi_{\beta}^{\la} \otimes e_{m}^{\sigma}
\right\rangle \\ &=& (2\pi)^{n} |\la|^{-n} e^{i \la (t+is)}
\sum_{\delta \in \mathbb{N}^n} \sum_{|\alpha| = |\gamma|}
\sum_{i,j=1}^{d_{\pi}} \sum_{p,q=1}^{d_{\nu}} \int_K
\int_{\mathbb{R}^{2n}} (f^{\la})^{pq}_{ij}(x',u')
\phi^{\la}_{\alpha \delta}(x',u') \phi^{\la}_{\delta \beta}(z,w) \\
&& \phi^{\sigma}_{ml}(k') dx' du' ~ \eta_{\alpha \gamma}^{\la}(k')
\phi_{ij}^{\pi}(e^{-iH}k^{-1}k') \phi_{pq}^{\nu}(e^{-iH}k^{-1})
dk' \eeas so that \beas && (2\pi)^{-2n} |\la|^{2n} \left\|
\rho_{\sigma}^{\la}(z,w,\tau,g)
\widehat{f}(\la,\sigma)(\phi_{\gamma}^{\la} \otimes
e_{l}^{\sigma}) \right\|^2 \\ &=& \sum_{\beta \in \mathbb{N}^n}
\sum_{1 \leq m \leq d_{\sigma}} e^{-2\la s} \left|
\sum_{p,q=1}^{d_{\nu}} \sum_{\delta \in \mathbb{N}^n}
\sum_{|\alpha| = |\gamma|} \sum_{i,j=1}^{d_{\pi}}
\left(\int_{\mathbb{R}^{2n}} (f^{\la})^{pq}_{ij}(x',u')
\phi^{\la}_{\alpha \delta}(x',u') dx' du'\right) \right. \\ &&
\left. \left(\int_K \phi^{\sigma}_{ml}(k') \eta_{\alpha
\gamma}^{\la}(k') \phi_{ij}^{\pi}(e^{-iH}k^{-1}k') dk'\right)
\phi_{pq}^{\nu}(e^{-iH}k^{-1}) \phi^{\la}_{\delta \beta}(z,w)
\right|^2, \eeas and summing over $\gamma$ and $l$ we have \beas
&& (2\pi)^{-2n} |\la|^{2n} \left\| \rho_{\sigma}^{\la}(z,w,\tau,g)
\widehat{f}(\la,\sigma) \right\|_{HS}^2 \\ &=& \sum_{\gamma, \beta
\in \mathbb{N}^n} \sum_{1 \leq l,m \leq d_{\sigma}} e^{-2\la s}
\left| \sum_{p,q=1}^{d_{\nu}} \sum_{\delta \in \mathbb{N}^n}
\sum_{|\alpha| = |\gamma|} \sum_{i,j=1}^{d_{\pi}}
\left(\int_{\mathbb{R}^{2n}} (f^{\la})^{pq}_{ij}(x',u')
\phi^{\la}_{\alpha \delta}(x',u') dx' du'\right) \right. \\ &&
\left. \left(\int_K \phi^{\sigma}_{ml}(k') \eta_{\alpha
\gamma}^{\la}(k') \phi_{ij}^{\pi}(e^{-iH}k^{-1}k') dk'\right)
\phi_{pq}^{\nu}(e^{-iH}k^{-1}) \phi^{\la}_{\delta \beta}(z,w)
\right|^2. \eeas Using Plancherel theorem for $K$ we derive that
\beas && (2\pi)^{-2n} |\la|^{2n} \sum_{\sigma \in \widehat{K}}
d_{\sigma} \left\| \rho_{\sigma}^{\la}(z,w,\tau,g)
\widehat{f}(\la,\sigma) \right\|_{HS}^2 \\ &=& e^{-2\la s}
\sum_{\gamma, \beta \in \mathbb{N}^n} \int_K
\Bigl|\sum_{p,q=1}^{d_{\nu}} \phi_{pq}^{\nu}(e^{-iH}k^{-1})
\sum_{\delta \in \mathbb{N}^n} \phi^{\la}_{\delta \beta}(z,w)
\sum_{|\alpha| = |\gamma|} \sum_{i,j=1}^{d_{\pi}} \left\langle
(f^{\la})^{pq}_{ij}, \overline{\phi^{\la}_{\alpha \delta}}
\right\rangle \\ && \eta_{\alpha \gamma}^{\la}(k')
\phi_{ij}^{\pi}(e^{-iH}k^{-1}k')\Bigl|^2 dk'. \eeas Applying the
same arguments as in (\ref{pr1}) and change of variables $k' \ra
kk'$ we obtain that the above equals \beas && e^{-2\la s}
\sum_{\alpha, \beta \in \mathbb{N}^n} \int_K
\left|\sum_{p,q=1}^{d_{\nu}} \phi_{pq}^{\nu}(e^{-iH}k^{-1})
\sum_{\delta \in \mathbb{N}^n} \phi^{\la}_{\delta \beta}(z,w)
\sum_{i,j=1}^{d_{\pi}} \left\langle (f^{\la})^{pq}_{ij},
\overline{\phi^{\la}_{\alpha \delta}} \right\rangle
\phi_{ij}^{\pi}(e^{-iH}k') \right|^2 dk' \\ &=& e^{-2\la s}
\sum_{\alpha, \beta \in \mathbb{N}^n} \int_K
\Bigl|\sum_{p,q=1}^{d_{\nu}} \phi_{pq}^{\nu}(e^{-iH}k^{-1})
\sum_{\delta \in \mathbb{N}^n} \phi^{\la}_{\delta \beta}(z,w)
\sum_{i,j,b=1}^{d_{\pi}} \left\langle (f^{\la})^{pq}_{ij},
\overline{\phi^{\la}_{\alpha \delta}} \right\rangle \\ &&
\phi_{ib}^{\pi}(e^{-iH}) \phi_{bj}^{\pi}(k') \Bigl|^2 dk'. \eeas
Now, using Schur's orthogonality relations we get that \beas &&
(2\pi)^{-2n} |\la|^{2n} \sum_{\sigma \in \widehat{K}} d_{\sigma}
\left\| \rho_{\sigma}^{\la}(z,w,\tau,g) \widehat{f}(\la,\sigma)
\right\|_{HS}^2 \\ &=& \frac{e^{-2\la s}}{d_{\pi}} \sum_{\alpha,
\beta \in \mathbb{N}^n} \sum_{j,b=1}^{d_{\pi}} \left|
\sum_{p,q=1}^{d_{\nu}} \phi_{pq}^{\nu}(e^{-iH}k^{-1}) \sum_{\delta
\in \mathbb{N}^n} \phi^{\la}_{\delta \beta}(z,w)
\sum_{i=1}^{d_{\pi}} \left\langle (f^{\la})^{pq}_{ij},
\overline{\phi^{\la}_{\alpha \delta}} \right\rangle
\phi_{ib}^{\pi}(e^{-iH}) \right|^2. \eeas Hence we have \bea
\label{p1} \begin{array}{rcll} && \ds{(2\pi)^{-2n} \sum_{\sigma
\in \widehat{K}} d_{\sigma} \int_{\mathbb{R}} \left\|
\rho_{\sigma}^{\la}(z,w,\tau,g) \widehat{f}(\la,\sigma)
\right\|_{HS}^2 |\la|^{n} d\la} \\ &=& \ds{\int_{\mathbb{R}}
\sum_{\alpha, \beta \in \mathbb{N}^n} \sum_{j,b=1}^{d_{\pi}}
\Bigl|\sum_{p,q=1}^{d_{\nu}} \phi_{pq}^{\nu}(e^{-iH}k^{-1})
\sum_{\delta \in \mathbb{N}^n} \phi^{\la}_{\delta \beta}(z,w)
\sum_{i=1}^{d_{\pi}} \left\langle (f^{\la})^{pq}_{ij},
\overline{\phi^{\la}_{\alpha \delta}} \right\rangle}
\\ && \ds{ \phi_{ib}^{\pi}(e^{-iH})\Bigl|^2 \frac{e^{-2\la s}}{d_{\pi}}
|\la|^{-n} d\la.}
\end{array} \eea

We have obtained an expression for one part of Lemma \ref{lem}.
Now, looking at the other part, we have \beas &&
f\left((x,u,t,k)^{-1} (x',u',t',k')\right) \\ &=&
f\left(k^{-1}(x'-x,u'-u), t'-t-\frac{1}{2}(u \cdot x' - x \cdot
u'),k^{-1}k' \right) \\ &=& \sum_{i,j=1}^{d_{\pi}}
\sum_{p,q=1}^{d_{\nu}} f^{pq}_{ij}
\left(x'-x,u'-u,t'-t-\frac{1}{2}(u \cdot x' - x \cdot u') \right)
\phi_{pq}^{\nu}\left(k^{-1}\right)
\phi_{ij}^{\pi}\left(k^{-1}k'\right). \eeas Since $f$ is
holomorphic on $\mathbb{HM}_{\mathbb{C}},$ each $f^{pq}_{ij}$ also
have a holomorphic extension to $\mathbb{HM}_{\mathbb{C}}.$ For
$(z,w,\tau,g) = (x+iy,u+iv,t+is,ke^{iH}) \in \mathbb{C}^n \times
\mathbb{C}^n \times \mathbb{C} \times G,$
we get \beas && f\left((z,w,\tau,g)^{-1} (x',u',t',k')\right) \\
&=& \sum_{i,j=1}^{d_{\pi}} \sum_{p,q=1}^{d_{\nu}} f^{pq}_{ij}
\left(x'-z,u'-w,t'-\tau-\frac{1}{2}(w \cdot x' - z \cdot u')
\right) \phi_{pq}^{\nu}\left(e^{-iH}k^{-1}\right) \\
&& \phi_{ij}^{\pi}\left(e^{-iH}k^{-1}k'\right). \eeas Taking the
$L^2$-norm with respect to $k'$ and applying change of variables
$k' \ra kk'$ and Schur's orthogonality relations, we obtain \beas
&& \int_K \left| f\left((z,w,\tau,g)^{-1} (x',u',t',k')\right)
\right|^2 dk'
\\ &=& \int_K \Bigl|\sum_{i,j,l=1}^{d_{\pi}} \sum_{p,q=1}^{d_{\nu}} f^{pq}_{ij}
\left(x'-z,u'-w,t'-\tau-\frac{1}{2}(w \cdot x' - z \cdot u')
\right) \phi_{pq}^{\nu}\left(e^{-iH}k^{-1}\right) \\
&& \phi_{il}^{\pi}\left(e^{-iH}\right)
\phi_{lj}^{\pi}\left(k'\right)\Bigl|^2 dk' \\ &=&
\frac{1}{d_{\pi}} \sum_{j,l=1}^{d_{\pi}}
\Bigl|\sum_{i=1}^{d_{\pi}} \sum_{p,q=1}^{d_{\nu}} f^{pq}_{ij}
\left(x'-z,u'-w,t'-\tau-\frac{1}{2}(w \cdot x' - z \cdot u')
\right) \phi_{pq}^{\nu}\left(e^{-iH}k^{-1}\right) \\
&& \phi_{il}^{\pi}\left(e^{-iH}\right) \Bigl|^2 . \eeas Now,
integrating over $t'$ we derive that \beas && \int_{\mathbb{R}}
\int_K \left| f\left((z,w,\tau,g)^{-1} (x',u',t',k')\right)
\right|^2 dk' dt' \\ &=& \frac{1}{d_{\pi}} \sum_{j,l=1}^{d_{\pi}}
\int_{\mathbb{R}} \Bigl|\sum_{i=1}^{d_{\pi}}
\sum_{p,q=1}^{d_{\nu}} (f^{pq}_{ij})^{\la}(x'-z,u'-w)
\phi_{pq}^{\nu}\left(e^{-iH}k^{-1}\right)
\phi_{il}^{\pi}\left(e^{-iH}\right) \Bigl|^2 \\ && e^{2 \la
(-s-\frac{1}{2}(v \cdot x' - y \cdot u'))} d\la. \eeas So we have
\beas && \int_{\mathbb{R}^{2n}} \int_{\mathbb{R}} \int_K \left|
f\left((z,w,\tau,g)^{-1} (x',u',t',k')\right) \right|^2 dk' dt' dx' du' \\
&=& \frac{1}{d_{\pi}} \sum_{j,l=1}^{d_{\pi}} \int_{\mathbb{R}}
e^{-2 \la s} \int_{\mathbb{R}^{2n}} \Bigl|\sum_{i=1}^{d_{\pi}}
\sum_{p,q=1}^{d_{\nu}} (f^{pq}_{ij})^{\la}(x'-z,u'-w)
\phi_{pq}^{\nu}\left(e^{-iH}k^{-1}\right)
\phi_{il}^{\pi}\left(e^{-iH}\right) \Bigl|^2 \\ && e^{\la (u'
\cdot y - v' \cdot x)} dx' du' d\la. \eeas From using (\ref{schr})
it follows that $$ \overline{\phi^{\la}_{\alpha \delta}(x,u)} =
\phi^{\la}_{\delta \alpha}(-x,-u).$$ So, we can expand
$(f^{pq}_{ij})^{\la}$ in terms of the orthonormal basis
$\overline{\phi^{\la}_{\alpha \delta}}$ to get $$
(f^{pq}_{ij})^{\la}(x,u) =  \sum_{\alpha, \delta \in \mathbb{N}^n}
\left\langle (f^{\la})^{pq}_{ij}, \overline{\phi^{\la}_{\alpha
\delta}} \right\rangle \phi^{\la}_{\delta \alpha}(-x,-u)$$ and
hence we have $$ (f^{pq}_{ij})^{\la}(x'-z,u'-w) =  \sum_{\alpha,
\delta \in \mathbb{N}^n} \left\langle (f^{\la})^{pq}_{ij},
\overline{\phi^{\la}_{\alpha \delta}} \right\rangle
\phi^{\la}_{\delta \alpha}(z-x',w-u').$$ Again, using (\ref{schr})
and the orthonormality of $\phi^{\la}_{\beta}$ we obtain that the
above equals $$ |\la|^{-\frac{n}{2}} e^{\frac{i \la}{2} (z \cdot
u' - w \cdot x')} \sum_{\alpha, \beta, \delta \in \mathbb{N}^n}
\left\langle (f^{\la})^{pq}_{ij}, \overline{\phi^{\la}_{\alpha
\delta}} \right\rangle \phi^{\la}_{\delta \beta}(z,w)
\phi^{\la}_{\beta \alpha}(-x',-u').$$ Hence we get that \beas &&
\int_{\mathbb{R}^{2n}} \int_{\mathbb{R}} \int_K \left|
f\left((z,w,\tau,g)^{-1} (x',u',t',k')\right) \right|^2 dk' dt' dx' du' \\
&=& \frac{1}{d_{\pi}} \sum_{j,l=1}^{d_{\pi}} \int_{\mathbb{R}}
|\la|^{-n} e^{-2 \la s} \int_{\mathbb{R}^{2n}}
\Bigl|\sum_{i=1}^{d_{\pi}} \sum_{p,q=1}^{d_{\nu}} \sum_{\alpha,
\beta, \delta \in \mathbb{N}^n} \left\langle (f^{\la})^{pq}_{ij},
\overline{\phi^{\la}_{\alpha \delta}} \right\rangle
\phi^{\la}_{\delta \beta}(z,w) \phi^{\la}_{\beta \alpha}(-x',-u')
\\ && \phi_{pq}^{\nu}\left(e^{-iH}k^{-1}\right)
\phi_{il}^{\pi}\left(e^{-iH}\right) \Bigl|^2  dx' du' d\la \\
&=& \frac{1}{d_{\pi}} \sum_{j,l=1}^{d_{\pi}} \sum_{\alpha, \beta
\in \mathbb{N}^n} \int_{\mathbb{R}} |\la|^{-n} e^{-2 \la s}
\Bigl|\sum_{i=1}^{d_{\pi}} \sum_{p,q=1}^{d_{\nu}} \sum_{\delta \in
\mathbb{N}^n} \left\langle (f^{\la})^{pq}_{ij},
\overline{\phi^{\la}_{\alpha \delta}} \right\rangle
\phi^{\la}_{\delta \beta}(z,w) \\ &&
\phi_{pq}^{\nu}\left(e^{-iH}k^{-1}\right)
\phi_{il}^{\pi}\left(e^{-iH}\right) \Bigl|^2 d\la. \eeas From
(\ref{p1}) and above we obtain the required equality.

For the converse, it is enough to prove the holomorphicity of $f$
which in turn follows from the holomorphicity of
$(f^{\la})^{pq}_{ij}$ and the equality follows from the above
argument. Assume that
$$\sum_{\sigma \in \widehat{K}} d_{\sigma} \int_{\mathbb{R}}
\left\| \rho_{\sigma}^{\la}(z,w,\tau,g) \widehat{f}(\la,\sigma)
\right|_{HS}^2 |\la|^{n} d\la < \infty ~~~~ \forall ~~
(z,w,\tau,g) \in \mathbb{C}^n \times \mathbb{C}^n \times
\mathbb{C} \times G.$$ From the above it is clear that for every
$1 \leq l, j \leq d_{\pi}$ and $\la$ almost everywhere \beas
\sum_{\alpha, \beta \in \mathbb{N}^n} \left|\sum_{i=1}^{d_{\pi}}
\sum_{p,q=1}^{d_{\nu}} \sum_{\delta \in \mathbb{N}^n} \left\langle
(f^{\la})^{pq}_{ij}, \overline{\phi^{\la}_{\alpha \delta}}
\right\rangle \phi^{\la}_{\delta \beta}(z,w)
\phi_{pq}^{\nu}\left(e^{-iH}k^{-1}\right)
\phi_{il}^{\pi}\left(e^{-iH}\right) \right|^2 < \infty \eeas for
all $\ds{(z,w) \in \mathbb{C}^n \times \mathbb{C}^n}$ and $ke^{iH}
\in G.$ We can put $e^{iH}=I,$ the identity of the group to get
\beas \sum_{\alpha, \beta \in \mathbb{N}^n}
\left|\sum_{p,q=1}^{d_{\nu}} \sum_{\delta \in \mathbb{N}^n}
\left\langle (f^{\la})^{pq}_{lj}, \overline{\phi^{\la}_{\alpha
\delta}} \right\rangle \phi^{\la}_{\delta \beta}(z,w)
\phi_{pq}^{\nu}(k)\right|^2 < \infty \eeas for all $\ds{(z,w) \in
\mathbb{C}^n \times \mathbb{C}^n}$ and $k \in K.$ Integrating over
$K$ and using Schur's orthogonality relations we have \beas &&
\sum_{\alpha, \beta \in \mathbb{N}^n} \int_K
\left|\sum_{p,q=1}^{d_{\nu}} \sum_{\delta \in \mathbb{N}^n}
\left\langle (f^{\la})^{pq}_{lj}, \overline{\phi^{\la}_{\alpha
\delta}} \right\rangle
\phi^{\la}_{\delta \beta}(z,w) \phi_{pq}^{\nu}(k)\right|^2 dk \\
&=& \sum_{\alpha, \beta \in \mathbb{N}^n} \sum_{p,q=1}^{d_{\nu}}
\left|\sum_{\delta \in \mathbb{N}^n} \left\langle
(f^{\la})^{pq}_{lj}, \overline{\phi^{\la}_{\alpha \delta}}
\right\rangle \phi^{\la}_{\delta \beta}(z,w)\right|^2 < \infty.
\eeas Hence we derive that for each $1 \leq p, q \leq d_{\nu},$ $1
\leq l, j \leq d_{\pi}$ and $\ds{(z,w) \in \mathbb{C}^n \times
\mathbb{C}^n}$ \beas \sum_{\alpha, \beta \in \mathbb{N}^n}
\left|\sum_{\delta \in \mathbb{N}^n} \left\langle
(f^{\la})^{pq}_{lj}, \overline{\phi^{\la}_{\alpha \delta}}
\right\rangle \phi^{\la}_{\delta \beta}(z,w)\right|^2 < \infty.
\eeas Let $T$ be the maximal torus of $K \subseteq U(n).$ After a
conjugation by an element of $U(n)$ if necessary, we can consider
that $T \subseteq \mathbb{T}^n,$ the $n$-dimensional torus which
is the maximal torus of $U(n).$ Now, any element $k_{\theta} \in
T$ can be written as $e^{i \theta} = (e^{i \theta_1}, e^{i
\theta_2}, \cdots , e^{i \theta_n})$ where $\theta = (\theta_1,
\theta_2, \cdots, \theta_n).$ Notice that some of these
$\theta_j$'s may be $0$ depending on $T.$ Using the relation
(\ref{meta}) and the properties of the metaplectic representation,
we have
$$\phi^{\la}_{\alpha \delta} (k_{\theta} \cdot (x,u)) = e^{i(\delta - \alpha)
\cdot \theta} \phi^{\la}_{\alpha \delta}(x,u).$$ Moreover, for
each $\nu \in \widehat{K},$ $\nu|_{T}$ breaks up into at most
$d_{\nu}$ irreducible components, not necessarily distinct, which
we call $\nu_1, \nu_2, \cdots, \nu_m \in \mathbb{Z}^n$ (abuse of
notation) such that $\nu_a(e^{i \theta}) = e^{i \nu_a \cdot
\theta}$ where $1 \leq a \leq m \leq d_{\nu}.$ Choosing
appropriate basis elements, the matrix coefficients
$\phi_{ab}^{\nu}$ of $\nu$ satisfy $\phi_{ab}^{\nu}(e^{i \theta})
= \delta_{ab} e^{i \nu_a \cdot \theta}$ where $\delta$ is the
Kronecker delta. So we obtain \beas && \left\langle
(f^{\la})^{pq}_{lj}, \overline{\phi^{\la}_{\alpha \delta}}
\right\rangle \\ &=& \int_T \int_{\mathbb{R}^{2n}}
(f^{\la})^{pq}_{lj}(k \cdot (x,u)) \phi^{\la}_{\alpha \delta}(k
\cdot (x,u)) dx du dk \\ &=& \sum_{r=1}^{d_{\nu}} \int_T
\int_{\mathbb{R}^{2n}} (f^{\la})^{pr}_{lj}(x,u)
\phi_{qr}^{\nu}(e^{i \theta}) e^{i(\delta - \alpha) \cdot \theta}
\phi^{\la}_{\alpha \delta}(x,u) dx du dk \\ &=& \left\langle
(f^{\la})^{pq}_{lj}, \overline{\phi^{\la}_{\alpha ~ \alpha-\nu_q}}
\right\rangle \delta_{\delta, \alpha-\nu_q} .\eeas Hence we get
that for each $1 \leq p, q \leq d_{\nu},$ $1 \leq l, j \leq
d_{\pi}$ and $\ds{(z,w) \in \mathbb{C}^n \times \mathbb{C}^n}$
\beas \sum_{\alpha \in \mathbb{N}^n} \left|\left\langle
(f^{\la})^{pq}_{lj}, \overline{\phi^{\la}_{\alpha ~ \alpha-\nu_q}}
\right\rangle \right|^2 \left(\sum_{\beta \in \mathbb{N}^n}
\left|\phi^{\la}_{\alpha-\nu_q ~ \beta}(z,w)\right|^2 \right) <
\infty. \eeas From the orthonormality properties of
$\phi^{\la}_{\alpha \beta}$ it follows that \bea \label{eqn}
\sum_{\alpha \in \mathbb{N}^n} \left|\left\langle
(f^{\la})^{pq}_{lj}, \overline{\phi^{\la}_{\alpha+\nu_q ~ \alpha}}
\right\rangle \right|^2 \phi^{\la}_{\alpha \alpha}(2iy,2iv) <
\infty. \eea Now, using the above we want to prove the
holomorphicity of $(f^{\la})^{pq}_{ij}.$ We note that for $(z,w)
\in \mathbb{C}^{2n},$ $$ (f^{\la})^{pq}_{ij}(z,w) = \sum_{\alpha
\in \mathbb{N}^n} \left\langle (f^{\la})^{pq}_{lj},
\overline{\phi^{\la}_{\alpha+\nu_q ~ \alpha}} \right\rangle
\phi^{\la}_{\alpha ~ \alpha+\nu_q}(-z,-w) $$ if the sum converges.
Consider a compact set $M \subseteq \mathbb{C}^{2n}$ such that
$|y|^2 + |v|^2 \leq r^2$ where $(z,w)=(x+iy,u+iv).$ We know that
$$ \phi^{\la}_{\alpha \alpha}(2iy,2iv) = C e^{\la(|y|^2 + |v|^2)}
L_{\alpha}^{0}(-2\la(|y|^2 + |v|^2))$$ for any $y,v \in
\mathbb{R}^n$ where $\ds{L_{\alpha}^0(z) = \prod_{j=1}^n
L_{\alpha_j}^0 (\frac{1}{2} |z_j|^2)}.$ Since $\phi_{\alpha
\alpha}(2iy,2iv)$ has exponential growth and (\ref{eqn}) implies
holomorphicity of $(f^{\la})^{pq}_{ij}$ as in the previous
section.
\end{proof}

\vspace{1.2in}

\noindent \textit{Proof of Theorem \ref{thm}.}

To prove the theorem, it is enough to prove the orthogonality of
the components $\ds{f^{\nu}_{\pi}(x,u,t,k) =
\sum_{i,j=1}^{d_{\pi}} \sum_{p=1}^{d_{\nu}} f^{pp}_{ij}(x,u,t)
\phi_{ij}^{\pi}(k)}.$ For $\pi, \nu, \pi', \nu' \in \widehat{K},$
we have \beas && (2\pi)^{-2n} |\la|^{2n} \sum_{\sigma \in
\widehat{K}} d_{\sigma} \left\langle
\rho_{\sigma}^{\la}(z,w,\tau,g)
\widehat{f^{\nu}_{\pi}}(\la,\sigma),
\rho_{\sigma}^{\la}(z,w,\tau,g)
\widehat{f^{\nu'}_{\pi'}}(\la,\sigma) \right\rangle_{HS} \\
&=& (2\pi)^{-2n} |\la|^{2n} \sum_{\sigma \in \widehat{K}}
d_{\sigma} \sum_{\beta, \gamma \in \mathbb{N}^n} \sum_{1 \leq l,m
\leq d_{\sigma}} \left\langle \rho_{\sigma}^{\la}(z,w,\tau,g)
\widehat{f^{\nu}_{\pi}}(\la,\sigma)(\phi_{\gamma}^{\la} \otimes
e_{l}^{\sigma}) , \phi_{\beta}^{\la} \otimes e_{m}^{\sigma}
\right\rangle \\ && \overline{\left\langle
\rho_{\sigma}^{\la}(z,w,\tau,g)
\widehat{f^{\nu'}_{\pi'}}(\la,\sigma)(\phi_{\gamma}^{\la} \otimes
e_{l}^{\sigma}) , \phi_{\beta}^{\la} \otimes e_{m}^{\sigma}
\right\rangle} \eeas \beas &=& \sum_{\sigma \in \widehat{K}}
d_{\sigma} \sum_{\beta, \gamma \in \mathbb{N}^n} \sum_{1 \leq l,m
\leq d_{\sigma}} e^{- 2 \la s} \sum_{p,q=1}^{d_{\nu}}
\phi_{pq}^{\nu}(e^{-iH}k^{-1}) \sum_{p',q'=1}^{d_{\nu'}}
\overline{\phi_{p'q'}^{\nu'}(e^{-iH}k^{-1})} \sum_{|\alpha| =
|\gamma|, |\alpha'| = |\gamma|} \\
&& \sum_{\delta, \delta' \in \mathbb{N}^n} \sum_{i,j=1}^{d_{\pi}}
\sum_{i',j'=1}^{d_{\pi'}} \left\langle (f^{\la})^{pq}_{ij},
\phi^{\la}_{\alpha \delta} \right\rangle \overline{\left\langle
(f^{\la})^{p'q'}_{i'j'},\phi^{\la}_{\alpha' \delta'}
\right\rangle} ~ \phi^{\la}_{\delta \beta}(z,w) ~ \overline{\phi^{\la}_{\delta' \beta}(z,w)} \\
&& \left( \int_K  \phi^{\sigma}_{ml}(k') \eta_{\alpha
\gamma}^{\la}(k') \phi_{ij}^{\pi}(e^{-iH}k^{-1}k') dk' \right)
\overline{\left( \int_K  \phi^{\sigma}_{ml}(k_1) \eta_{\alpha'
\gamma}^{\la}(k_1) \phi_{i'j'}^{\pi'}(e^{-iH}k^{-1}k_1) dk_1
\right)}. \eeas By Schur's orthogonality relations we get that the
above equals \beas && \sum_{\beta, \gamma \in \mathbb{N}^n} e^{- 2
\la s} \sum_{p,q=1}^{d_{\nu}} \phi_{pq}^{\nu}(e^{-iH}k^{-1})
\sum_{p',q'=1}^{d_{\nu'}}
\overline{\phi_{p'q'}^{\nu'}(e^{-iH}k^{-1})} \sum_{|\alpha| =
|\gamma|, |\alpha'| = |\gamma|} \sum_{\delta, \delta' \in
\mathbb{N}^n} \sum_{i,j=1}^{d_{\pi}} \sum_{i',j'=1}^{d_{\pi'}} \\
&& \left\langle (f^{\la})^{pq}_{ij}, \phi^{\la}_{\alpha \delta}
\right\rangle \overline{\left\langle
(f^{\la})^{p'q'}_{i'j'},\phi^{\la}_{\alpha' \delta'}
\right\rangle} ~ \phi^{\la}_{\delta \beta}(z,w) ~
\overline{\phi^{\la}_{\delta' \beta}(z,w)} \int_K \eta_{\alpha \gamma}^{\la}(k') \phi_{ij}^{\pi}(e^{-iH}k^{-1}k') \\
&& \overline{\eta_{\alpha' \gamma}^{\la}(k')
\phi_{i'j'}^{\pi'}(e^{-iH}k^{-1}k')} dk' \eeas By arguments
similar to (\ref{pr1}) we have that the above equals \beas &&
\sum_{\alpha, \beta \in \mathbb{N}^n} e^{- 2 \la s}
\sum_{p,q=1}^{d_{\nu}} \phi_{pq}^{\nu}(e^{-iH}k^{-1})
\sum_{p',q'=1}^{d_{\nu'}}
\overline{\phi_{p'q'}^{\nu'}(e^{-iH}k^{-1})} \sum_{\delta, \delta'
\in \mathbb{N}^n} \sum_{i,j=1}^{d_{\pi}} \sum_{i',j'=1}^{d_{\pi'}}
\left\langle (f^{\la})^{pq}_{ij}, \phi^{\la}_{\alpha \delta}
\right\rangle \\
&& \overline{\left\langle
(f^{\la})^{p'q'}_{i'j'},\phi^{\la}_{\alpha' \delta'}
\right\rangle} ~ \phi^{\la}_{\delta \beta}(z,w) ~
\overline{\phi^{\la}_{\delta' \beta}(z,w)} \int_K
\phi_{ij}^{\pi}(e^{-iH}k^{-1}k') \overline{
\phi_{i'j'}^{\pi'}(e^{-iH}k^{-1}k')} dk' \eeas \beas &=&
\sum_{\alpha, \beta \in \mathbb{N}^n} e^{- 2 \la s}
\sum_{p,q=1}^{d_{\nu}} \phi_{pq}^{\nu}(e^{-iH}k^{-1})
\sum_{p',q'=1}^{d_{\nu'}}
\overline{\phi_{p'q'}^{\nu'}(e^{-iH}k^{-1})} \sum_{\delta, \delta'
\in \mathbb{N}^n} \sum_{i,j=1}^{d_{\pi}} \sum_{i',j'=1}^{d_{\pi'}}
\left\langle (f^{\la})^{pq}_{ij}, \phi^{\la}_{\alpha \delta}
\right\rangle \\
&& \overline{\left\langle
(f^{\la})^{p'q'}_{i'j'},\phi^{\la}_{\alpha' \delta'}
\right\rangle} ~ \phi^{\la}_{\delta \beta}(z,w) ~
\overline{\phi^{\la}_{\delta' \beta}(z,w)} ~ \sum_{b=1}^{d_{\pi}}
\sum_{b'=1}^{d_{\pi'}} \phi_{ib}^{\pi}(e^{-iH}) \overline{
\phi_{i'b'}^{\pi'}(e^{-iH})} \\ && \int_K \phi_{bj}^{\pi}(k')
\overline{ \phi_{b'j'}^{\pi'}(k')} dk' \\ &=& 0 \txt{ if } \pi
\ncong \pi'. \eeas Assume $\pi \cong \pi'.$ Then \beas &&
(2\pi)^{-2n} |\la|^{2n} \sum_{\sigma \in \widehat{K}} d_{\sigma}
\int_K \left\langle \rho_{\sigma}^{\la}(z,w,\tau,ke^iH)
\widehat{f^{\nu}_{\pi}}(\la,\sigma),
\rho_{\sigma}^{\la}(z,w,\tau,ke^{iH})
\widehat{f^{\nu'}_{\pi}}(\la,\sigma) \right\rangle_{HS} dk
\\ &=& \sum_{\alpha, \beta \in \mathbb{N}^n} e^{- 2 \la s}
\sum_{p,q,r=1}^{d_{\nu}} \phi_{pr}^{\nu}(e^{-iH})
\sum_{p',q',r'=1}^{d_{\nu'}}
\overline{\phi_{p'r'}^{\nu'}(e^{-iH})} \int_K \phi_{rq}^{\nu}(k)
\overline{\phi_{r'q'}^{\nu'}(k)} dk \sum_{\delta, \delta' \in
\mathbb{N}^n} \sum_{i,j,i',j'=1}^{d_{\pi}} \\
&& \left\langle (f^{\la})^{pq}_{ij}, \phi^{\la}_{\alpha \delta}
\right\rangle \overline{\left\langle
(f^{\la})^{p'q'}_{i'j'},\phi^{\la}_{\alpha' \delta'}
\right\rangle} \phi^{\la}_{\delta \beta}(z,w) ~
\overline{\phi^{\la}_{\delta' \beta}(z,w)} \int_K
\phi_{ij}^{\pi}(e^{-iH}k') \overline{
\phi_{i'j'}^{\pi}(e^{-iH}k')} dk' \\ &=& 0 \txt{ if } \nu \ncong
\nu'. \eeas This proves the orthogonality of one part. On the
other hand, for $\pi, \nu, \pi', \nu' \in \widehat{K}$ and $X =
\left(x'-z,u'-w,t'-\tau-\frac{1}{2}(w \cdot x' - z \cdot
u')\right),$ we have \beas && \int_K f\left((z,w,\tau,g)^{-1}
(x',u',t',k')\right)
\overline{f\left((z,w,\tau,g)^{-1} (x',u',t',k')\right)} dk' \\
&=& \sum_{i,j=1}^{d_{\pi}} \sum_{i',j'=1}^{d_{\pi'}}
\sum_{p,q=1}^{d_{\nu}} \sum_{p',q'=1}^{d_{\nu'}} f^{pq}_{ij}(X)
\overline{f^{p'q'}_{i'j'}(X)} \phi_{pq}^{\nu}\left(e^{-iH}k^{-1}\right)
\overline{\phi_{p'q'}^{\nu'}\left(e^{-iH}k^{-1}\right)} \\
&& \sum_{b=1}^{\pi} \sum_{b'=1}^{\pi'}
\phi_{ib}^{\pi}\left(e^{-iH}\right)
\overline{\phi_{i'b'}^{\pi}\left(e^{-iH}\right)} \int_K
\phi_{bj}^{\pi}(k') \overline{\phi_{b'j'}^{\pi'}(k')} dk' \\ &=& 0
\txt{ if } \pi \ncong \pi'. \eeas Assume $\pi \cong \pi'.$ Then
\beas && \int_K \int_K f\left((z,w,\tau,ke^{iH})^{-1}
(x',u',t',k')\right) \overline{f\left((z,w,\tau,ke^{iH})^{-1} (x',u',t',k')\right)} dk' dk \\
&=& \sum_{i,j,i',j'=1}^{d_{\pi}} \sum_{p,q,b=1}^{d_{\nu}}
\sum_{p',q',b'=1}^{d_{\nu'}} f^{pq}_{ij}(X)
\overline{f^{p'q'}_{i'j'}(X)} \phi_{pb}^{\nu}(e^{-iH})
\overline{\phi_{p'b'}^{\nu'}(e^{-iH})} \int_K \phi_{bq}^{\nu}(k)
\overline{\phi_{b'q'}^{\nu'}(k)} dk \\
&& \int_K \phi_{ij}^{\pi}\left(e^{-iH} k'\right)
\overline{\phi_{i'j'}^{\pi}\left(e^{-iH} k'\right)} dk'
\\ &=& 0 \txt{ if } \nu \ncong \nu'. \eeas

\vspace{0.5 in}

\textbf{Acknowledgement.} Theorem \ref{Gutzmer} is adapted from an
unpublished work of Ms. Jotsaroop Kaur. The author wishes to thank
her for allowing the use of the same in this paper. The author is
also grateful to Dr. E. K. Narayanan for his constant
encouragement and for the many useful discussions during the
course of this work.

\end{document}